\documentclass[11pt]{article}
\usepackage[T1]{fontenc}
\usepackage[utf8]{inputenc}
\usepackage[margin=1.0in]{geometry}
\usepackage{mathtools} 
\usepackage{amsthm}
\usepackage{amssymb} 
\usepackage{amsfonts, dsfont} 
\usepackage{graphicx} 
\usepackage{color} 
\usepackage{subcaption}
\usepackage{enumerate} 
\usepackage{fancybox} 
\usepackage{float}  
\usepackage{verbatim}
\usepackage{stmaryrd}
\usepackage{empheq}
\usepackage{booktabs}
\usepackage{mathrsfs}
\usepackage{tikz}
\usetikzlibrary{cd}
\usepackage{appendix}
\usepackage{ytableau}
\usepackage{authblk}

\newtheorem{prop}{Proposition}[section]
\newtheorem{lemma}{Lemma}[section]
\newtheorem{remark}{Remark}[section]
\newtheorem{coro}{Corollary}[section]
\newtheorem{definition}{Definition}[section]

\newcommand{\R}{\mathbb{R}}
\newcommand{\Z}{\mathbb{Z}}
\newcommand{\N}{\mathbb{N}}
\newcommand{\C}{\mathbb{C}}
\newcommand{\dd}{\mathrm{d}}

\newcommand{\tr}{\text{tr}}

\title{Geometric background for the Teukolsky equation revisited}
\author{Pascal Millet}
\affil{Université Grenoble Alpes, Institut Fourier,\\
100 rue des Maths, 38610 Gières, France\\
pascal.millet@univ-grenoble-alpes.fr}
\date{}
\makeindex
\begin{document}
\maketitle

\begin{abstract}
We present in detail the geometric framework necessary to understand the Teukolsky equation and we develop in particular the case of Kerr spacetime.
\end{abstract}
\tableofcontents

\section{Introduction}
The Teukolsky equation (introduced in \cite{teukolsky1973perturbations}) is a differential equation on spin weighted functions. The geometrical framework necessary to understand these objects includes spin geometry, the Newman Penrose formalism (see \cite{newman1962approach} and \cite{newman1963errata}) and the closely related GHP formalism (introduced in \cite{geroch1973space}). There are some good general presentations of these formalisms in the literature: \cite{chandrasekhar1998mathematical} (Newman Penrose formalism and spin geometry), \cite{penrose1984spinors} (Newman Penrose formalism and spin geometry) \cite{andersson2015spin}(spin geometry), \cite{daude2018asymptotic} section 2.4(spin geometry and GHP formalism), \cite{aksteiner2011linearized} section 2.1(GHP formalism), \cite{aksteiner2019new}(GHP formalism), \cite{harnett1990ghp} (geometrical definition of spin weighted functions close to the one we use here and GHP formalism) and \cite{curtis1978complex}(interpretation of spin weighted functions as complex line bundles).

The goal of these notes is to synthesize the minimal geometric background necessary to understand how spin weighted functions appear in the study of the Teukolsky equation and to provide detailed computations in the case of the Kerr spacetime. In particular, we explicit the link, between the abstract definition of spin weighted functions (relying on the Newman Penrose formalism) and the more concrete definition involving the Hopf bundle. This second point of view is used in \cite{dafermos2019boundedness}, \cite{shlapentokh2020boundedness} and \cite{da2020mode} in the context of analysis of the Teukolsky equation on a Kerr background but also in \cite{goldberg1967spin} and \cite{eastwood1982edth} in the context of spin weighted functions on the two dimensional sphere. For simplicity, we do not consider here the interaction of the GHP formalism with the conformal structure of the space-time but this point of view (developped by Araneda in \cite{araneda2018conformal}) enables to give even more geometrical meaning to the Teukoslky operator (with in particular a natural definition of the Teukolsky connection). We do not claim to prove new results in this paper but we think that this synthesis will be useful for analysts looking for a simple and detailed introduction to the geometric framework of the Teukolsky equation in the Kerr spacetime. 

In a first part, we introduce the general definitions spin weighted functions on a general Petrov D type space-time. In the second part, we compute the topology of the bundles in the case of a Kerr space-time. In the third part, we introduce (with details) the various connections and define the GHP operators. We choose a geometrical approach using principal connections which is not the most direct but is quite natural. We used the references cited in this introduction but we reformulate everything needed in an essentially self contained way (except for some punctual peripheral propositions).
\begin{comment}
\cite{chandrasekhar1998mathematical} chapter 10, section 102 (Spinor analysis and the spinorial basis of the Newman Penrose formalism).
\cite{penrose1984spinors} section 4.12
\cite{eastwood1982edth} dans le cas de la fibration de Hopf sur $\mathbb{S}^2$
\cite{goldberg1967spin} dans le cas de la fibration de Hopf sur $\mathbb{S}^2$
\cite{newman1962approach} Article de Newman Penrose
\cite{newman1963errata} errata du précédent
\cite{beyer2014numerical} section geometric preliminaries (cité par Dafermos). Fibration de Hopf (opérateur sur la 3 sphère).
\cite{shlapentokh2020boundedness} section 2.2, utilise la déf de Dafermos
%\cite{ma2020uniform} Utilise (sans définir) la notion de poids spinoriel
\cite{da2020mode} uses the definition introduced in Dafermos
\cite{geroch1973space} formalisme GHP.
\cite{aksteiner2011linearized} Brève introduction à GHP section 2.1
\cite{andersson2019stability} Brève introduction.
\cite{harnett1990ghp} Source géométrique pour la def des spin weighted function et formalisme GHP.
%\cite{kerr1963gravitational} Papier original de Kerr
\end{comment}
\subsection*{Acknowledgements}
I am very greatful to my phD advisor Dietrich Häfner for numerous fruitful discussions about this work. I also thank Peter Hintz and Bernard F. Whiting for their comments and corrections on the first version of the paper which helped to improve it. Some of the computations in this paper have been verified using the python library sympy (see \cite{10.7717/peerj-cs.103}).

\subsection{Conventions and notations}
In this paper, we use the sign convention $(+,-,-,-)$ for Lorentzian metric. For example, the Minkowski metric on $\R^4$ will be given by the matrix $\eta = \begin{pmatrix}1 & 0 & 0 & 0\\ 0 & -1 & 0 & 0 \\ 0 & 0 &-1 &0 \\ 0 & 0 & 0 & -1 \end{pmatrix}$.

If we have a Cartesian product $A\times B$ we denote by $\text{pr}_A$ the projection on $A$.

Let $k\in \Z$ and $n\in \N$. We denote by $[k]_n$ the image of $k$ by the projection $\Z \rightarrow \Z/n\Z$. More generally, we will note $[x]$ the image of $x$ by the projection map when we have a quotient space.

If $E$ is a smooth vector bundle (real or complex) over a manifold $\mathcal{M}$, we denote by $E'$ the dual vector bundle (for every $x\in \mathcal{M}$, $(E')_x$ is the space of linear forms on $E_x$). We denote by $\Gamma(E)$ the set of smooth sections of $E$.

If $\nabla$ is a linear connection acting on a complex vector bundle, we sometimes want to compute $\nabla_X$ where $X$ is a section of the complexified tangent space. In this case, we simply extend the connection by $\C$-linearity, in other words, we define $\nabla_X:=\nabla_{\Re(X)} + i\nabla_{\Im(X)}$.

\section{Definition of the bundles}
Let $\mathcal{M}$ be a space-time (4 dimensional time and space oriented Lorentzian manifold, globally hyperbolic). Since the space-time is globally hyperbolic, there is a global time function and a 1+3 decomposition. A time orientation is the choice of a global timelike vector field and a space orientation is an orientation of the spatial part in the 1+3 decomposition.
\begin{definition}
For $x\in \mathcal{M}$, an oriented basis of $T_x\mathcal{M}$ is a basis $(b_1, ... ,b_4)$ with $b_1$ future oriented and timelike and $(b_2, b_3, b_4)$ spacelike with their spatial part (in the 1+3 decomposition) being an oriented basis.
\end{definition}
Note that a more intrinsic (but equivalent) way of defining time orientation would be: a continuous choice of a connected component of the lightcone on each point. A more intrinsic way of defining a space orientation would be: a smooth choice of a connected component of the set of triple $(b_2,b_3,b_4)$ of independent spacelike vector fields at each point.

The principal motivation to introduce spin weighted functions is the following: In the case where the space-time is of Petrov type D, we would like to find a global Newman Penrose null tetrad containing two vectors in the principal null directions (see definition \ref{principalNullVector}). This tetrad can be associated (up to a sign ambiguity) with a spin frame which is adapted to the geometry. Indeed, it ensures the vanishing of some associated spin coefficients and therefore simplifies the component expression of the Dirac operator. We can even hope to reduce tensor equations to scalar decoupled equations for some components.
However, it is in general not possible to find such a tetrad (and such spin frame) globally for topological reasons (see for example the computations on Kerr). However, it is possible to choose a global tetrad "up to some complex factor" (and similarly for the spin frame). Rigorously, this "almost tetrad" is defined as the smooth bundle of all four vectors satisfying the Newman Penrose conditions (or the bundle of all normalized spin frame associated to it). The analog of components in this "almost tetrad" can then be defined for co-tensors (and co-spinors in the almost spin-frame). Such components can be interpreted as sections of some complex line bundle $\mathcal{B}(s,w)$. Such sections of $\mathcal{B}(s,w)$ are called spin weighted functions.

\subsection{Spin structure and frame of oriented orthonormal frames}
We know by \cite{geroch1970spinor} that $\mathcal{M}$ admits a spin structure. In other words, if we denote by $\mathfrak{O}$ the $SO^+(1,3)$ principal bundle of oriented orthonormal frames on $\mathcal{M}$ (and $\pi_{\mathfrak{O}}$ is the associated projection), there exists a $SL(2, \C)$ principal bundle $\pi_{\mathfrak{S}}:\mathfrak{S} \rightarrow \mathcal{M}$ and a double covering $p: \mathfrak{S} \rightarrow \mathfrak{O}$ such that the following diagram commutes
\[
\begin{tikzcd}[column sep = 3em]
\mathfrak{S} \arrow{r}{p} \arrow[d, "\pi_{\mathfrak{S}}"] & \mathfrak{O} \arrow[ld, "\pi_{\mathfrak{O}}"] \\
\mathcal{M}
\end{tikzcd}
\]
and such that for all $g\in SL(2, \C)$ and $x\in \mathfrak{S}$, $p(x \cdot g) = p(x)\cdot \tilde{p}(g)$ where $\tilde{p}: SL(2, \C) \rightarrow SO^+(1, 3)$ is given by the Weyl representation
\[\tilde{p}: \begin{cases}
SL(2,\C) \rightarrow SO^+(1,3) \\
M \mapsto \left( x \in \R^4 \mapsto i_2^{-1} M\frac{1}{\sqrt{2}}\begin{pmatrix} x_0 + x_3 & x_1+ix_2 \\x_1-ix_2 & x_0-x_3 \end{pmatrix}M^*\right) 
\end{cases}
\]
where \[i_2:\begin{cases}\C^4 \rightarrow \mathcal{M}_2(\C) \\
z \mapsto \frac{1}{\sqrt{2}}\begin{pmatrix}z_0 + z_3 & z_1 + iz_2 \\ z_1 - iz_2 & z_0 - z_3 \end{pmatrix}.
\end{cases}
\]
We also identify $\C^2\otimes\overline{\C}^2$ and  $\mathcal{M}_2(\C)$ via the linear isomorphisms 
\[i_1:\begin{cases}\C^2 \otimes \overline{\C}^2 \rightarrow \mathcal{M}_2(\C) \\
a \otimes \overline{b} \mapsto a \overline{b}^T.
\end{cases}
\]
For later use, we define $i_0 := i_2^{-1}\circ i_1$ the identification between $\C^2\otimes\overline{\C^2}$ and $\C^4$.
Using the identification $i_0$, we have for every $g \in SL(2,\C)$: \begin{equation} \rho(g) \otimes \overline{\rho}(g) = \mu \circ \tilde{p}(g) \label{coresSpinVect}\end{equation} where $\rho$ is the canonical representation of $SL(2, \C)$ on $\C^2$ and $\mu$ is the canonical representation of $SO^+(1,3)$ on $\C^4$ (the identification $i_0$ is implicit in the equality).

\subsection{Spinor bundles}

\subsubsection{Vector bundle associated to a principal bundle}
Given a principal bundle $\pi:\mathcal{E}\rightarrow \mathcal{B}$ and a representation $\rho$ of the structure group $G$ on some (real or complex) vector space $V$, we can form the associated vector bundle $\mathcal{F}:=\mathcal{E}\times V / \sim$ where $(e, v)\sim (e', v')$ if there exists $g\in G$ such that $e\cdot g = e'$ and $\rho(g^{-1})(v) = v'$. The vector space structure is given on each fiber by $\lambda[(e,v)] = [(e, \lambda v)]$ and $[(e,v)] + [(e,v')] = [(e,v+v')]$ which does not depend on the choice of $(e,v)$ in the class. If $\phi: \pi^{-1}U \rightarrow U\times G$ is a local trivialization (in the sense of principal bundles) of $\mathcal{E}$, then $\phi_{\mathcal{F}}:x \in \mathcal{F} \mapsto ((\pi(x),z) \text{ such that }[(\phi^{-1}( \pi(x),1),z)] = x)$ is a local trivialization of $\mathcal{F}$. In these notes, $\phi_\mathcal{F}$ will be called the trivialization associated to $\phi$. 
If we apply the previous construction to the bundle of oriented orthonormal frames with the canonical representation of $SO^+(1,3)$ on $\C^4$, we obtain the complexified tangent space on $\mathcal{M}$.

\begin{remark}
If $A$ and $B$ are complex vector bundles associated to $\mathcal{E}$ for the action $\rho_A$ and $\rho_B$, then $A\otimes B$ is naturally isomorphic to the bundle associated to $\mathcal{E}$ for the action $\rho: g \mapsto \rho_A(g)\otimes\rho_B(g)$ and $\overline{A}$ is naturally isomorphic to the bundle associated to $\mathcal{E}$ for the action $\overline{\rho}_A$
\end{remark}

\subsubsection{Vector bundle associated with the spin structure}
Applying the previous construction to the spin structure and the canonical representation of $SL(2,\C)$, we get a complex vector bundle of rank 2 over $\mathcal{M}$ called the spinor bundle $\mathcal{S}$. If we chose a local trivialization $\Phi$ of  the spin structure, we deduce an associated local trivialization $\Phi_\mathcal{S}$ of $\mathcal{S}$. 
\begin{prop}
Given a local trivialization $\Phi: \pi_{\mathfrak{S}}^{-1}(U)\rightarrow U\times SL(2,\C)$ of $\mathfrak{S}$, there exists a unique local trivialization (on the same open set) $\Psi: \pi_{\mathfrak{O}}^{-1}(U)\rightarrow U\times SO^+(1,3)$ of $\mathfrak{O}$ such that $\Psi \circ p \circ \Phi^{-1} = Id\times\tilde{p}$. We say that $\Phi$ and $\Psi$ are compatible.
\end{prop}
\begin{remark}[Notation]
In these notes, if we have a Cartesian product $A\times B$ we denote by $pr_A$ the projection on $A$ and $pr_B$ the projection on $B$.
\end{remark}
\begin{proof}
A necessary condition on $\Psi$ is $\Psi \circ p = (Id \times \tilde{p})\circ \Phi$.
Therefore, the uniqueness is a direct consequence of the surjectivity of $p: \pi_{\mathfrak{S}}^{-1}(U)\rightarrow \pi_{\mathfrak{O}}^{-1}(U)$.
To construct $\Psi$, we have to show that if $p(x) = p(y)$, then $(Id \times \tilde{p})\circ \Phi (x) = (Id \times \tilde{p})\circ \Phi (y)$. Assume $x,y\in \mathfrak{S}$ are such that $p(x) = p(y)$. In particular $\pi_{\mathfrak{S}}(x) = \pi_{\mathfrak{S}}(y)$. By definition of a principal $SL(2,\C)$ bundle, there exists $g \in SL(2, \C)$ such that $y = x\cdot g$. By definition of $p$, $p(x) = p(x\cdot g) = p(x)\cdot\tilde{p}(g)$ and $\tilde{p}(g) = 1$ (the action of $SO^+(1,3)$ on each fiber of $\mathfrak{O}$ is free). Moreover, because $\Phi$ is a trivialization, $\Phi(x\cdot g) = (\pi_\mathfrak{S}(y), pr_{SL(2, \C)}(\Phi(y)) g)$. Now, \[(Id \times \tilde{p})\circ \Phi(x\cdot g) = (\pi_\mathfrak{S}(x), \tilde{p}(pr_{SL(2, \C)}(\Phi(x))) \tilde{p}(g)) = (Id \times \tilde{p})\circ \Phi(x).\]
So $\Psi$ is well defined. Moreover, using that $(Id \times \tilde{p})\circ \Phi$ is surjective and $(Id \times \tilde{p})\circ \Phi(x) = (Id \times \tilde{p})\circ \Phi(y)$ if and only if $p(x) = p(y)$ (the only if part comes from the fact that $(Id\times\tilde{p})\circ \Phi(x)$ and $p$ have exactly two antecedents), $\Psi$ is bijective. Since $p$ is a smooth covering map and $\Psi\circ p$ is a local diffeomorphism, we obtain that $\Psi$ is a local diffeomorphism so it is a diffeomorphism. We also have that for $g\in SO^+(1,3)$, $\Psi(x\cdot g) = (\pi_\mathfrak{O}(x), pr_{SO^+(1,3)}(\Psi(x))g)$. As a consequence $\Psi$ is a local trivialization of $\mathfrak{O}$.
\end{proof}
\begin{remark}
A local trivialization of a principal bundle can be defined by a local smooth section of the bundle (we ask that this section is the constant map equal to the neutral element of the group when written in the local trivialization). Two local sections $s_1$ of $\mathfrak{S}$ and $s_2$ of $\mathfrak{O}$ define compatible local trivializations if and only if $s_2 = p\circ s_1$. This fact could have been used to prove the proposition as well.
\end{remark}
\begin{remark}
\label{compatibleReverse}
It is not true in general that for any local trivialization $\Psi$ of $\mathfrak{O}$, there exists a compatible local trivialization of $\mathfrak{S}$. However it is true locally (there exists exactly 2 such compatible local trivializations corresponding to the two local lifts of the section $x \mapsto \Psi^{-1}(x,1)$ through $p$ ). 
\end{remark}
\begin{remark}
\label{changeOfMap}
If $\Phi\circ\Phi'^{-1}(x, g) = (x, gf(x))$ is a change of local trivialization of $\mathcal{E}$ (with $f: U\rightarrow G$ a smooth map), then $\Phi_{\mathcal{F}}\circ\Phi_{\mathcal{F}}'^{-1}(x, v) = (x, \rho(f(x))(v))$.
\end{remark}

Two compatible trivializations enable to reduce locally the picture of spinors and vectors on $\mathcal{M}$ to the picture of spinors and vectors on Minkowski space. In particular, we have an analog on $\mathcal{M}$ of the previously defined identification between $\C^2\otimes \overline{\C^2}$ and $\C^4$:
\begin{prop}
There exists a unique isomorphism $j$ of complex vector bundles between $\mathcal{S}\otimes \overline{\mathcal{S}}$ and $T_{\C}\mathcal{M}$ such that, given any couple of compatible local trivializations $(\Phi,\Psi)$ on $U$, $\Psi_{\mathcal{T_\C\mathcal{M}}}\circ j\circ \Phi_{\mathcal{S}\otimes\overline{\mathcal{S}}}^{-1}$ is exactly $(Id_U, i_0)$ where $i_0$ is the identification between $\C^2\otimes\overline{\C}^2$ and $\C^4$
\end{prop}
\begin{proof}
The uniqueness is obvious since compatible local trivializations cover $\mathcal{M}$. For the existence, we define $j$ locally on each open set $U$ associated to a compatible local trivialization and we have to check that all the definitions agree when they overlap. Let $(\Phi, \Psi)$ and $(\Phi', \Psi')$ be two couples of compatible local trivializations on the same open set $U$. Let $g_1: U \rightarrow SL(2,\C)$ smooth be such that $\Phi \circ \Phi'^{-1}(x, g) = (x, gg_1(x))$) and $g_2: U \rightarrow SO^+(1,3)$ smooth be such that $\Psi \circ \Psi'^{-1}(x, g) = (x,gg_2(x))$. Using the compatibility (therefore $p = \Psi^{-1} \circ Id\times \tilde{p}\circ \Phi = \Psi'^{-1}\circ Id\times\tilde{p}\circ \Phi'$), we get for all $x\in U$ and all $g\in SL(2, \C)$ $(x,\tilde{p}(g)) = (x, \tilde{p}(gg_1(x))g_2(x)^{-1})$  and we deduce $\tilde{p}\circ g_1 = g_2$. Then $\Psi_{\mathcal{T_\C\mathcal{M}}}\circ\Psi'^{-1}_{\mathcal{T_\C\mathcal{M}}}\circ(Id \times i_0)\circ\Phi'_{\mathcal{S}\otimes\overline{\mathcal{S}}}\circ\Phi^{-1}_{\mathcal{S}\otimes\overline{\mathcal{S}}}(x,v) = \left(x, \mu(\tilde{p}(g_1(x)))(i_0(\rho(g_1(x)^{-1})\otimes \overline{\rho}(g_1(x)^{-1})(v)))\right)$ (where $\rho$ and $\mu$ are the same as in \eqref{coresSpinVect} and we used remark \ref{changeOfMap} to compute the changes of associated trivialization). But \eqref{coresSpinVect} states exactly that $\mu(\tilde{p}(g_1(x)))\circ i_0 \circ(\rho(g_1(x)^{-1})\otimes \overline{\rho}(g_1(x)^{-1})) = i_0$. Therefore $\Psi_{\mathcal{T_\C\mathcal{M}}}\circ\Psi'^{-1}_{\mathcal{T_\C\mathcal{M}}}\circ(Id \times i_0)\circ\Phi'_{\mathcal{S}\otimes\overline{\mathcal{S}}}\circ\Phi^{-1}_{\mathcal{S}\otimes\overline{\mathcal{S}}}(x,v) = (Id \times i_0)(x, v)$ and the proposition is proved.
\end{proof}
\begin{remark}
In the following, we identify $\mathcal{S}\otimes \overline{\mathcal{S}}$ and $T_\C\mathcal{M}$ using the map $j$ implicitly. Note that if $m = j(a\otimes \overline{b}) \in T_\C\mathcal{M}$, then $\overline{m} = j(b\otimes \overline{a})$. This comes from the analogous property on the map $i_0$.
\end{remark}

\subsubsection{Symplectic form on spinors}
\begin{prop}
\label{defEpsilon}
There exists a unique symplectic form $\epsilon \in \mathcal{S}'\wedge\mathcal{S}'$ such that for all local trivializations $\Phi$ of $\mathfrak{S}$, we have $\epsilon_x (\Phi_{\mathcal{S}}^{-1} (x,v), \Phi_{\mathcal{S}}^{-1} (x,w)) = \det (v,w)$.
\end{prop}
\begin{proof}
The uniqueness is obvious. To prove the existence, we define $\epsilon$ locally and check that all the definitions agree. Let $\Phi$ and $\Phi'$ be two local trivializations of $\mathfrak{S}$ on some open set $U$. We denote by $g: U \rightarrow SL(2,\C)$ the smooth map such that $\Phi\circ\Phi'^{-1} (x, h) = (x, hg(x))$. Then $\det( \rho(g(x)^{-1})(v), \rho(g(x)^{-1})(w)) = \det(g(x)^{-1}(v,w)) = \det(g(x)^{-1})\det(v,w)=\det(v,w)$. The proposition is proved.
\end{proof}

\begin{prop}
\label{relEpsilonG}
We have the following equality for all $x\in \mathcal{M}$, for all $a, c \in \mathcal{S}_x$ and $\overline{b},\overline{d}\in \overline{\mathcal{S}}_x$:
\[
g(a\otimes \overline{b}, c\otimes \overline{d}) = \epsilon(a,c)\overline{\epsilon(b,d)}
\]
\end{prop}
\begin{proof}
We fix a pair of compatible trivializations $(\Phi, \Psi)$. The equality to prove in the associated trivializations is
\[
\forall a,c \in \C^2, \forall \overline{b},\overline{d} \in \overline{\C}^2, \eta (i_0(a \otimes \overline{b}), i_0(c\otimes \overline{d}))= \det(a,c)\overline{\det(b,d)}
\]
where $\eta = \begin{pmatrix} 1 & 0 & 0 & 0\\ 0 & -1 & 0 & 0 \\ 0& 0 & -1 & 0\\ 0 & 0 & 0 & -1\end{pmatrix}$ is the Minkowski metric on $\C^4$. We check that this equality is true using the explicit definition of $i_0$.
\end{proof}

\subsection{Bundle of normalized spin frame along the null directions}
From now on, we assume that $\mathcal{M}$ is a Petrov-type $D$ (see the definition below) Ricci-flat spacetime.
\begin{definition} \label{principalNullVector}
A null vector $l\in T_x\mathcal{M}$ is called principal if for all $a,b \in T_x\mathcal{M}$ such that $g(l,a) = g(l,b)=0$, we have $W(l,a,l,b)= 0$ where $W$ is the Weyl tensor. We say that the vector $l$ is principal of multiplicity at least $2$ if for all $a,b \in T_x\mathcal{M}$ such that $g(l,a)= 0$, $W(l,a,l,b) =0$. This characterization of principal vectors can be found in \cite{o2014geometry}, proposition 5.5.5. 
Alternative equivalent definitions and additional properties of principal null directions are also provided in \cite{o2014geometry}.
 \end{definition}

\begin{definition}
We define a Petrov type D space-time as a space-time $\mathcal{M}$ such that for all $x_0\in \mathcal{M}$, there exists $l$ and $n$ independent (therefore non vanishing) null vector fields defined on a neighborhood $U$ of $x_0$ such that:
\begin{itemize}
\item For all $x\in U$, the set of principal vector fields at $T_x\mathcal{M}$ is exactly $\R l(x) \cup \R n(x)$
\item For all $x\in U$, $l(x)$ and $n(x)$ are principal null vectors of multiplicity (at least) $2$.
(in fact $l$ and $n$ are of multiplicity exactly 2, see \cite{o2014geometry}, chapter 5)
\end{itemize}
\end{definition}

The property of principal null directions on Petrov-type D Ricci-flat space-time that we need here is the following:
\begin{prop}
If we define $n$ and $l$ as in the previous definition, they are pregeodesic and shear-free. The shear of a pregeodesic vector field $l$ at $x_0$ with respect to $X,Y$ orthonormal family of $l(x_0)^\perp$ is defined as $\frac{1}{2}\left(g(\nabla_Y l, Y)-g(\nabla_X l, X)\right)+\frac{i}{2}\left(g(\nabla_Y l, X)+g(\nabla_X l, Y)\right)$ (see definition 5.7.1 in \cite{o2014geometry})
\end{prop}
\begin{proof}
%Let $\mathcal{M}$ be a Petrov-type D Ricci-flat space-time. Let $x_0 \in \mathcal{M}$.
%To prove the first part of the proposition, we fix $x_0 \in \math$
See proposition 5.9.2 in \cite{o2014geometry}.
\end{proof}

We define the following subset of $\mathcal{S}\times_{\mathcal{M}}\mathcal{S}$:
\begin{align*}
\mathcal{A}:=&\cup_{x\in\mathcal{M}}\left\{(o,\iota)\in \mathcal{S}_x\times\mathcal{S}_x: j(o\otimes\overline{o}) \text{ and } j(\iota \otimes \overline{\iota}) \text{ are independent, future oriented}\right.\\
&\left.\text{along null principal directions and }\epsilon(o,\iota)=1\right\}
\end{align*}

The bundle $\mathcal{A}$ is naturally endowed with a canonical $\C^*$ right action:
\[ (o, \iota)\cdot z = (zo, z^{-1}\iota) \]
We also have a $\Z/4\Z$ right action given by the map (image of the generator of $\Z/4\Z$):
$(o, \iota) \mapsto (i \iota, i o)$

By combining the two actions (performing the action of $\C^*$ first) , we get a right action of $\C^*\rtimes_f \Z/4\Z$ (with $f : [1]_4 \mapsto (z\mapsto z^{-1})$) on $\mathcal{A}$. However, this action is not free and we can quotient by the stabilizer of any point in $\mathcal{A}$ which is the normal subgroup $H:=\left<(-1, [2]_4)\right>$ to get a free action of the group $G_{\mathcal{A}}:=(\C^*\rtimes_f \Z/4\Z )/ H$. Topologically, $\C^*\rtimes_f \Z/4\Z$ is simply $\C^*\times \Z/4\Z$ which is homeomorphic to four disjoint copies of $\C^*$. The quotient by $H$ identifies $\C^*\times\left\{[0]_4\right\}$ with $\C^*\times\left\{[2]_4\right\}$ and $\C^*\times\left\{[1]_4\right\}$ with $\C^*\times \left\{[3]_4\right\}$ so $G_{\mathcal{A}}$ is homeomorphic to two disjoint copies of $\C^*$ (more precisely the two connected components are $\C^*_0:=\left\{[(z, [0]_4)], z\in \C^*\right\}$ and $\C^*_1:=\left\{[(z, [1]_4)], z\in \C^*\right\}$, $\C^*_0$ being the connected component of the neutral element). Note that $\C^*_0$ is a normal subgroup of $G_{\mathcal{A}}$ isomorphic to $\C^*$.

\begin{prop}
\label{principalA}
The set $\mathcal{A}$ is a smooth submanifold of $\mathcal{S}\times_{\mathcal{M}}\mathcal{S}$. Moreover, $\mathcal{A}$ is a $G_{\mathcal{A}}$ principal bundle (with projection map $\pi_{\mathcal{A}} = \pi_{\mathcal{S}\times_{\mathcal{M}}\mathcal{S}}{}_{|_\mathcal{A}}$).
\end{prop}

%\begin{remark}
%Here we are given a left action on $\mathcal{A}$ and we usually require a right action for principal bundles (for example in the definition of the associate vector bundle). We can define a right action given by $a\cdot(z,u) := (z,u)^{-1}\cdot a$. If we compose each local trivialization for the left action by $(z,u)\mapsto (z,u)^{-1}$ we obtain local trivializations for the right action.
%\end{remark}

\begin{proof}
Let $x\in \mathcal{M}$. There exists an open neighborhood $U$ of $x$ such that there exists an oriented orthonormal tetrad $(e_0,e_1, e_2, e_3)$ with $e_0+e_3$ and $e_0-e_3$ are future oriented along the null principal directions. This tetrad gives a local trivialization $\Psi$ of $\mathfrak{O}$. Then, even taking a smaller neighborhood, we can assume that there exists a local trivialization $\Phi$ of $\mathfrak{S}$ on $U$ such that $\Phi$ and $\Psi$ are compatible (see remark \ref{compatibleReverse}). The set $\tilde{\mathcal{A}}:=\Phi_{\mathcal{S}\times_{\mathcal{M}}\mathcal{S}} (\pi_{\mathcal{A}}^{-1}(U)) $ is given by $U\times A$ where $A\subset \C^2\times \C^2$ is defined as the set of $(o, \iota)\in \C^2\times\C^2$ such that:
\begin{align*}
\begin{cases}
 i_0(o\otimes \overline{o})&= (\lambda , 0, 0, \lambda)\\
i_0(\iota \otimes \overline{\iota}) &= (\mu, 0, 0, -\mu)\\
\end{cases}\\
\text{or} \\
\begin{cases}
 i_0(\iota\otimes \overline{\iota})&= (\lambda , 0, 0, \lambda)\\
i_0(o \otimes \overline{o}) &= (\mu, 0, 0, -\mu)\\
\end{cases}\\
\det(o, \iota) = 1
\end{align*}
where $\lambda$ and $\mu$ are real positive numbers. We find the following parametrization for $A$:
\[
\alpha:
\begin{cases}
G_{\mathcal{A}} \rightarrow A \\
[(z, [0]_4)] \mapsto \left(\begin{pmatrix}z \\ 0 \end{pmatrix}, \begin{pmatrix} 0 \\ z^{-1}\end{pmatrix}\right) \\
[(z, [1]_4)] \mapsto \left(\begin{pmatrix}0 \\ iz^{-1} \end{pmatrix}, \begin{pmatrix} iz \\ 0 \end{pmatrix}\right)
\end{cases}
\]
where for $k \in \left\{0,1,2,3\right\}$, $[(z, [k]_4)]$ is the class $(z, [k]_4)H$ in $G_{\mathcal{A}}$. This map is a proper injective immersion in $\C^2\times\C^2$ and its image is $A$. The fact that the image is included in $A$ can be checked directly with the definitions. The other inclusion is proved by solving the system defining $A$. For example, if $o = \begin{pmatrix}o_0\\o_1\end{pmatrix}$ and $\lambda \in (0,+\infty)$, the condition $i_0(o\otimes \overline{o}) = (\lambda, 0, 0, \lambda)$ rewrites (by definition of $i_0$): \[ \left(\frac{\sqrt{2}}{2}\left(\lvert o_0 \rvert^2+\lvert o_1\rvert^2\right), \sqrt{2}\Re(o_0\overline{o_1}), \sqrt{2}\Im(o_0\overline{o_1}), \frac{\sqrt{2}}{2}\left(\lvert o_0 \rvert^2-\lvert o_1\rvert^2\right)\right)=(\lambda, 0, 0, \lambda)\]
This equality holds if and only if $o_1 = 0$ and $\left\vert o_0\right\rvert^2 = 2\lambda$. The other cases are very similar.

This prove that $A$ is a submanifold of $\C^2\times\C^2$ and we deduce that $\mathcal{A}$ is a submanifold of $\mathcal{S}\times_{\mathcal{M}} \mathcal{S}$. Moreover, we check (using the definition of the action) that the maps $\Phi_{\mathcal{S}\times_{\mathcal{M}}\mathcal{S}}^{-1}\circ\left(Id_U\times\alpha\right)$ (defined around each point $x\in\mathcal{M}$) endow $\mathcal{A}$ with a structure of $G_{\mathcal{A}}$ principal bundle over $\mathcal{M}$.
\end{proof}

\subsection{Bundle of oriented Newman-Penrose null tetrads along principal null directions}

\begin{definition}
\label{orientationOfTriad}
Let $(l,n,m)\in T_{\C,x}\mathcal{M}^3$  be such that $(n,l,m,\overline{m})$ is a null basis of $T_{\C,x}\mathcal{M}$  with $l$ and $n$ real and future oriented. We say that $(n,l,m)$ is oriented if $(\frac{n+l}{\sqrt{2}}, \Re(m), -\Im(m), \frac{l-n}{\sqrt{2}})$ is oriented.
\end{definition}
We define the following subset of $T_\C\mathcal{M}^3$:
\begin{align*}
\mathcal{N}:=&\cup_{x\in\mathcal{M}}\left\{(l,n,m)\in \left(T_{\C,x}\mathcal{M}\setminus \left\{0\right\}\right)^4 / g(l,l)=g(n,n)=g(m,m)=0, l=\overline{l}, n=\overline{n},\right. \\
&\text{$l$ and $n$ are independent principal and future oriented}, g(l,m)=g(n,m)=0, g(m,\overline{m})= -1, g(l,n)=1,\\
& \left.\text{$(l,n,m)$ is oriented in the sense of definition \ref{orientationOfTriad}}  \vphantom{\left(T_{\C,x}\mathcal{M}\setminus \left\{0\right\}\right)^4}\right\}
\end{align*}
In particular for $(l,n,m)\in \mathcal{N}_x$, $(l,n,m,\overline{m})$ is a basis of $T_{\C,x}\mathcal{M}$.
The set $\mathcal{N}$ is endowed with a canonical $\C^*$ right action:
\[ (l,n,m)\cdot z = (|z|l, |z|^{-1}n, \frac{z}{|z|}m) \]
We also have a $\Z/2\Z$ right action defined by the involution: 
\[ (l,n,m) \mapsto (n,l,\overline{m})\]

Combining these two actions (performing the action of $\C^*$ first), we get a right action of $\C^*\rtimes_g \Z/2\Z$ on $\mathcal{N}$ where $g([1]_2)(z) = z^{-1}$.

\begin{prop}
\label{principalN}
$\mathcal{N}$ is a smooth submanifold of $T_{\C}\mathcal{M}^3$ and is a $\C^*\rtimes_g \Z/2\Z$ principal bundle (for the action previously defined) with projection map $\pi_{\mathcal{N}}:= (\pi_{T_{\C}\mathcal{M}^3})_{|_{\mathcal{N}}}$.
\end{prop}

%\begin{remark}
%As for $\mathcal{A}$, we can define a right action on $\mathcal{N}$ given by $a\cdot(z,u) = (z,u)^{-1}\cdot a$.
%\end{remark}

\begin{proof}
Let $x\in \mathcal{M}$, there exists an open neighborhood $U$ of $x$ such that there exists an oriented orthonormal tetrad $(e_0,e_1, e_2, e_3)$ with $e_0+e_3$ and $e_0-e_3$ are future oriented along the null principal directions. We denote by $\Psi$ the corresponding trivialization of $\mathfrak{O}$. Then $\Psi_{T_\C \mathcal{M}^3}(\mathcal{N}) = U\times N$ with $N$ subset of $(\C^4)^3$ of $(l,n,m)$ such that:
\begin{align*}
\begin{cases}
l &= \frac{\sqrt{2}}{2}\begin{pmatrix} \lambda\\ 0 \\ 0 \\ \lambda \end{pmatrix}\\
n &= \frac{\sqrt{2}}{2}\begin{pmatrix} \lambda^{-1} \\ 0 \\ 0 \\ - \lambda^{-1} \end{pmatrix} \\
m &= \frac{\sqrt{2}}{2}\begin{pmatrix} 0 \\ e^{i\theta} \\ -ie^{i\theta} \\ 0 \end{pmatrix} 
\end{cases}
\text{ or }
\begin{cases}
n &= \frac{\sqrt{2}}{2}\begin{pmatrix} \lambda\\ 0 \\ 0 \\ \lambda \end{pmatrix}\\
l &= \frac{\sqrt{2}}{2}\begin{pmatrix} \lambda^{-1} \\ 0 \\ 0 \\ - \lambda^{-1} \end{pmatrix} \\
m &= \frac{\sqrt{2}}{2}\begin{pmatrix} 0 \\ e^{i\theta} \\ ie^{i\theta} \\ 0 \end{pmatrix} 
\end{cases}
\end{align*}
where $\lambda$ is a positive real and $\theta$ is any real.

We have the following injective proper immersion (in $(\C^4)^3$)
\[
\beta:\begin{cases}
\C^*\rtimes_g \Z/2\Z \rightarrow N \\
(z, [0]_2) \mapsto \left(\frac{1}{\sqrt{2}}\begin{pmatrix}|z| \\ 0 \\ 0 \\ |z|\end{pmatrix}, \frac{1}{\sqrt{2}}\begin{pmatrix}|z|^{-1} \\ 0 \\ 0 \\-|z|^{-1}\end{pmatrix}, \frac{1}{\sqrt{2}}\begin{pmatrix} 0 \\ \frac{z}{|z|} \\ -i\frac{z}{|z|} \\ 0\end{pmatrix}\right)\\
(z, [1]_2) \mapsto \left(\frac{1}{\sqrt{2}}\begin{pmatrix}|z|^{-1} \\ 0 \\ 0 \\ -|z|^{-1}\end{pmatrix}, \frac{1}{\sqrt{2}}\begin{pmatrix}|z| \\ 0 \\ 0 \\|z|\end{pmatrix}, \frac{1}{\sqrt{2}}\begin{pmatrix} 0 \\ \frac{\overline{z}}{|z|} \\ i\frac{\overline{z}}{|z|} \\ 0\end{pmatrix}\right)
\end{cases}
\]
This map has image $N$ so $N$ is a submanifold of $(\C^4)^3$ and therefore, $\mathcal{N}$ is a submanifold of $T_\C\mathcal{M}^3$. Moreover, the maps $\beta$ (around each $x\in\mathcal{M}$) define the structure of a $\C^*\rtimes_g \Z/2\Z$ principal bundle on $\mathcal{N}$.
\end{proof}

We have a natural map:
\[
d: \begin{cases} \mathcal{A} \rightarrow \mathcal{N} \\
(o, \iota) \mapsto (o\otimes \overline{o}, \iota\otimes \overline{\iota}, o\otimes \overline{\iota})
\end{cases}
\]
\begin{remark}
The fact that the map is well defined follows from proposition \ref{relEpsilonG} and the following remark about orientation: If $(o, \iota)\in \mathcal{A}_x$, then if we fix a couple of compatible local trivializations $(\Phi, \Psi)$around $x$, $M:=\left(pr_{\C^2}(\Phi_{\mathcal{S}}(o)), pr_{\C^2}(\Phi_{\mathcal{S}}(\iota))\right) \in SL(2,\C)$. Then by a change of compatible local trivializations (associated to the smooth maps $x\in U \mapsto M^{-1}$ and $x\in U \mapsto \tilde{p}(M^{-1})$), we can assume that $M = \begin{pmatrix} 1 & 0 \\ 0 & 1\end{pmatrix}$. It follows that $pr_{\C^4}\Psi_{T_\C\mathcal{M}}(o\otimes\overline{o}) = \frac{\sqrt{2}}{2}\begin{pmatrix} 1 \\ 0 \\ 0 \\1 \end{pmatrix}$, $pr_{\C^4}\Psi_{T_\C\mathcal{M}}(\iota\otimes\overline{\iota}) = \frac{\sqrt{2}}{2}\begin{pmatrix} 1 \\ 0 \\ 0 \\-1 \end{pmatrix}$ and $pr_{\C^4}\Psi_{T_\C\mathcal{M}}(o\otimes\overline{\iota}) = \frac{\sqrt{2}}{2}\begin{pmatrix} 0 \\ 1 \\ -i \\0 \end{pmatrix}$. Since $\Psi$ is a trivialization of $\mathfrak{O}$ (oriented orthonormal bases), $\Psi_{T_\C\mathcal{M}}^{-1}$ sends the canonical basis of $\R^4$ on an oriented orthonormal basis of $T\mathcal{M}$. Therefore $(o \otimes \overline{o}, \iota \otimes \overline{\iota}, o\otimes \overline{\iota})$  is oriented in the sens of definition \ref{orientationOfTriad}.
\end{remark}

\subsubsection{Properties of $d$}
\begin{prop}
\label{doubleCover}
The map $d$ is a double covering map from $\mathcal{A}$ to $\mathcal{N}$ and for $[(z,[u]_4)]\in G_{\mathcal{A}}$ we have 
\[ d(a\cdot [(z,[u]_4)]) = d(a)\cdot (z^2,[u]_2)\]
\end{prop}
\begin{remark}
\label{kernelMap}
The map $\begin{cases}\C^*\rtimes_f \Z/4\Z \rightarrow \C^*\rtimes_g \Z/2\Z\\
(z, [u]_4)\mapsto (z^2, [u]_2)\end{cases}$ is well defined and is a group morphism (the key point is that $f([u]_4)=g([u]_2)$). The normal subgroup $H$ is included in the kernel so we have a group morphism $\begin{cases}G_{\mathcal{A}}\rightarrow \C^*\rtimes_g \Z/2\Z\\ [(z, [u]_4)]\mapsto (z^2, [u]_2)\end{cases}$. The map is surjective and its kernel is the discrete normal subgroup $\left\{[(1, [0]_4)], [(-1, [0]_4)]\right\}$. Therefore, it is a Lie group double covering map.
\end{remark}
\begin{proof}
Let $x\in \mathcal{M}$, we define $U$, $\Phi$ and $\Psi$ as in the proof of proposition \ref{principalA}.Then the map $\Phi_{T_\C\mathcal{M}^3}\circ d\circ\left(\Psi^{-1}_{\mathcal{S}\times\mathcal{S}}\right)_{|_{U\times A}} = Id_U\times \hat{d}$ where
\[
\hat{d}:\begin{cases}
A \rightarrow N \\
\left(\begin{pmatrix}z \\ 0\end{pmatrix}, \begin{pmatrix}0 \\ z^{-1}\end{pmatrix}\right) \mapsto \left( \frac{1}{\sqrt{2}}\begin{pmatrix}|z|^2 \\ 0 \\ 0 \\ |z|^2\end{pmatrix}, \frac{1}{\sqrt{2}}\begin{pmatrix}|z|^{-2} \\ 0 \\ 0\\ -|z|^{-2} \end{pmatrix}, \frac{1}{\sqrt{2}}\begin{pmatrix}0\\ \frac{z^2}{|z|^2} \\ -i\frac{z^2}{|z|^2} \\0\end{pmatrix}\right)\\
\left(\begin{pmatrix}0 \\ iz^{-1}\end{pmatrix}, \begin{pmatrix}iz \\ 0\end{pmatrix}\right) \mapsto \left( \frac{1}{\sqrt{2}}\begin{pmatrix}|z|^{-2} \\0 \\0 \\-|z|^{-2}\end{pmatrix}, \frac{1}{\sqrt{2}}\begin{pmatrix} |z|^2 \\ 0 \\ 0 \\|z|^{2}\end{pmatrix}, \frac{1}{\sqrt{2}}\begin{pmatrix}0\\ \frac{\overline{z}^2}{|z|^2} \\i\frac{\overline{z}^2}{|z|^2} \\0 \end{pmatrix}\right)
\end{cases}
\]
($A$ and $N$ are as defined in the proofs of propositions \ref{principalA} and \ref{principalN}. The submanifold $A$ is parametrized by $G_{\mathcal{A}}$ using the maps $\alpha$ and $B$ is parametrized by $\C^*\rtimes_g\Z/2\Z$ using the map $\beta$)
With this expression, we deduce the expression of the map $\beta^{-1} \circ\hat{d}\circ \alpha$:
\[\beta^{-1} \circ \hat{d}\circ \alpha:
\begin{cases}
G_{\mathcal{A}} \rightarrow \C^*\rtimes_g \Z/2\Z \\
[(z, [u]_4)] \mapsto (z^2, [u]_2)
\end{cases}
\]
To conclude the proof, note that $Id_U\times \left(\beta^{-1} \circ\hat{d}\circ \alpha\right)$ is the expression of the map $d$ in local trivializations for the structure of principal bundles (the same trivializations used to define the structure in the proof of propositions \ref{principalA} and \ref{principalN}).
\end{proof}

\subsection{Vector bundles associated to $\mathcal{A}$ and $\mathcal{N}$ and spin weighted functions}
In this section, we have to make an assumption on the topology of $\mathcal{A}$. We assume that $\mathcal{A}$ has exactly two connected components. This assumption is true for a large variety of spacetimes of interest thanks to the following proposition:
\begin{prop}
\label{simplyConnected}
If $\mathcal{M}$ is simply connected, then $\mathcal{A}$ and $\mathcal{N}$ have exactly two connected components.
\end{prop}
\begin{proof}
We sketch the proof for $\mathcal{A}$, then  see remark \ref{compAcompN} to deduce the result for $\mathcal{N}$.
We know that each fiber of $\mathcal{A}$ has two connected components diffeomorphic to $\C^*$. Then we deduce (since $\mathcal{M}$ is connected) that $\mathcal{A}$ has one or two connected components. By contradiction, assume that $\mathcal{A}$ has only one connected component. Take $x\in \mathcal{M}$, we call $\C^*_0$ and $\C^*_1$ the two connected components of $\mathcal{A}_x$. By the hypothesis, there exists a continuous path $\gamma : [0,1]\rightarrow \mathcal{A}$ such that $\gamma(0) = 1_0$ and $\gamma(1)= 1_1$. Then $\pi_{\mathcal{A}}\circ \gamma$ is a loop on $\mathcal{M}$. But $\mathcal{M}$ is simply connected. Therefore there exists a homotopy $f:[0,1]\times[0,1]\rightarrow \mathcal{M}$ between $\pi_{\mathcal{A}}\circ \gamma$ and the constant loop $t \in [0,1]\mapsto x$ such that $f_t(0) = x$ and $f_t(1)=x$ for all $t\in[0,1]$. But $\mathcal{A}$ is a fiber bundle over $\mathcal{M}$ so it has the homotopy lifting property and we can find a lift $\tilde{f}$ of $f$ such that $\tilde{f}_0 = \gamma$. But the concatenation of $t\in [0,1]\mapsto \tilde{f}_t(0)$, $t \in [0,1]\mapsto \tilde{f}_1(t)$ and $t\in [0,1] \mapsto \tilde{f}_{1-t}(1)$ is a continuous path with value in $\mathcal{A}_x$ joining $1_0$ and $1_1$ which is a contradiction. 
\end{proof}

\begin{remark}
\label{compAcompN}
If $\mathcal{A}$ has two connected components $\mathcal{A}_0$ and $\mathcal{A}_1$, then $d(\mathcal{A}_0)$ and $d(\mathcal{A}_1)$ are disjoint. Indeed, if $d(x) = d(y)$ either $x=y$ or, by proposition \ref{doubleCover} (and the end of remark \ref{kernelMap}), we have $y = x\cdot[(-1, [0]_4)]$ and the continuous path $t\mapsto x\cdot[(e^{it\pi}, [0]_4)]$ joins $x$ and $y$. Therefore $\mathcal{N}$ has also two connected components given by $d(\mathcal{A}_0)$ and $d(\mathcal{A}_1)$ (Indeed, these two sets are connected, and there exists no continuous path from one to the other otherwise we could lift this path to a path between $\mathcal{A}_0$ and $\mathcal{A}_1$).
\end{remark}

\begin{prop}
\label{globalln}
We assume that $\mathcal{M}$ is connected.
$\mathcal{N}$ has two connected components if and only if there exists two global smooth null future oriented vector fields $l$ and $n$ such that at each point $x\in \mathcal{M}$, $l(x)$ and $n(x)$ are independent and principal.
\end{prop}

\begin{proof}
We assume that $\mathcal{N}$ has two connected components. We choose one that we call $\mathcal{N}_0$. Then if $(l,n,m), (l',n',m')\in \mathcal{N}_0$ with $\pi_\mathcal{N}(l,n,m) = \pi_\mathcal{N}(l',n',m')$, there exists $z \in \C^*$ such that $(l,n,m) = (l',n',m')\cdot(z,[0]_2) = (|z|l', |z|^{-1}n', \frac{z}{|z|}m')$. In particular $l$ and $l'$ are positively collinear as well as $n$ and $n'$. Then any convex combination of $l$ and $l'$ is principal null, the same for $n$ and $n'$ and the two are independent. This remark enables us to construct global vector fields $n$ and $l$ from local sections of $\mathcal{N}_0$ using a partition of unity. 
If we assume the existence of the global vector fields $l$ and $n$, then we have the two connected components:\newline
$\mathcal{N}_0 := \left\{ (u, v, m): u \text{ and } l \text{ are collinear } \text{ and } v \text{  and  } n \text{ are collinear }\right\}$ and\newline 
$\mathcal{N}_1 := \left\{ (u, v, m): u \text{ and } n \text{ are collinear } \text{ and } v \text{  and  } l \text{ are collinear }\right\}$.
\end{proof}

From now we assume that $\mathcal{A}$ has two connected components. We choose one component that we call $\mathcal{A}_0$ and we define $\mathcal{N}_0 := d(\mathcal{A}_0)$ (which is one of the two connected components of $\mathcal{N}$ according to remark \ref{compAcompN})). The other connected component is called $\mathcal{A}_1$ (and $\mathcal{N}_1$). This choice defines an additional notion of orientation which corresponds to an ordering of the principal null directions ($\mathcal{A}_0$ is the subset of oriented elements of $\mathcal{A}$).

The right action of the subgroup $\C^*_0:=\left\{ [(z, [0]_4)], z\in \C^*\right\}$ of $G_{\mathcal{A}}$ gives a structure of $\C^*$ principal bundle on $\mathcal{A}_0$. Similarly, $\left\{(z,[0]_2), z\in \C^*\right\}$ gives a structure of $\C^*$ principal bundle on $\mathcal{N}_0$. The action of $[(1, [1]_4)]$ induces a diffeomorphism between $\mathcal{A}_0$ and $\mathcal{A}_1$ (similarly, the action of $(1,[1]_2)$ induces a diffeomorphism between $\mathcal{N}_0$ and $\mathcal{N}_1$).

\begin{remark}
Since $\C^*$ is commutative, the right action is also a left action (and we use both notations in the following).
\end{remark}

Let $w,s\in \frac{1}{2}\Z$, we have the following representation of $\C^*$:
\[\rho_{s,w}:
\begin{cases}
\C^* \rightarrow GL(\C) \\
z \mapsto (a\mapsto z^{-w-s}\overline{z}^{-w+s}a)
\end{cases}
\]
We define the bundle $\mathcal{B}(s,w)$ as the complex line bundle associated to $\mathcal{A}_0$ (with the right action) and the representation $\rho_{s,w}$. We have a natural identification between sections of $\mathcal{B}(s,w)$ and the set of complex valued functions $f$ defined on $\mathcal{A}_0$ such that for all $z\in \C^*$ \begin{equation}f(a \cdot z) = z^{w+s}\overline{z}^{w-s}f(a). \label{spinWeight}\end{equation} The identification is given by: $f \mapsto (x\mapsto [(a, f(a))]$ where $a$ is any element of $(\mathcal{A}_0)_x)$. We call a spin weighted function with weights $(s,w)$ any section of $\mathcal{B}(s,w)$ or equivalently (with the identification) any function on $\mathcal{A}_0$ satisfying \eqref{spinWeight}. We denote by $W_{(s,w)}$ the set of spin weighted functions.
\begin{remark}
\label{productSpinWeighted}
We have the following canonical identification $\mathcal{B}(s+s',w+w') = \mathcal{B}(s,w)\otimes\mathcal{B}(s', w')$.
\end{remark}
\begin{remark} The number $s$ is called the spin weight and the number $w$ is called the boost weight in \cite{penrose1984spinors} section 4.12. 
\end{remark}

We call $\boldsymbol{o}$ (resp. $\boldsymbol{\iota}$) the first (resp. second) projection from $\mathcal{A}_0$ to $\mathcal{S}$ and we define $\boldsymbol{l} := \boldsymbol{o}\otimes\overline{\boldsymbol{o}}$, $\boldsymbol{n} := \boldsymbol{\iota}\otimes\overline{\boldsymbol{\iota}}$ and $\boldsymbol{m} = \boldsymbol{o}\otimes\overline{\boldsymbol{\iota}}$. Note that thanks to the map $j$ the maps $\boldsymbol{l}$, $\boldsymbol{m}$ and $\boldsymbol{n}$ can be seen as $T_{\C}\mathcal{M}$ valued maps.
We can think of these maps as generalized spin frame and generalized tetrad. Note that due to the relation $\boldsymbol{o}(u\cdot z) = z\boldsymbol{o}(u)$ (resp. $\boldsymbol{\iota}(u\cdot z) = z^{-1}\boldsymbol{\iota}(u)$) for $u \in \mathcal{A}_0$, we can identify $\boldsymbol{o}$ (resp. $\boldsymbol{\iota}$) with a smooth section of the bundle $\mathcal{B}\left(\frac{1}{2},\frac{1}{2}\right)\otimes \mathcal{S}$ (resp. $\mathcal{B}\left(-\frac{1}{2}, -\frac{1}{2}\right)\otimes \mathcal{S}$). 

The following proposition is the main reason of why we are interested in spin weighted functions.
\begin{prop}[Spin weighted components of cospinors]
\label{correspondence}
We denote by $\mathcal{S}_{a,b} := \otimes^{a}\mathcal{S}' \otimes^{b}\overline{\mathcal{S}}'$. There is a bijection $F$ between the set of sections of $\mathcal{S}_{a,b}$ and the set $\prod_{I\subset \llbracket 1, a\rrbracket, J\subset \llbracket 1, b\rrbracket} W_{(|I|-|J|+\frac{b-a}{2}, |I|+|J|-\frac{a+b}{2})}$ given by $F:u \mapsto \prod_{I\subset \llbracket 1, a\rrbracket, J\subset \llbracket 1, b\rrbracket} u_{I,J}$ where for $y\in \mathcal{A}$, $u_{I,J}(y)= u(g_1(y), \dots, g_a(y), h_1(y), \dots , h_b(y))$ with $g_i = \boldsymbol{o}$ if $i\in I$, $g_i = \boldsymbol{\iota}$ if $i\notin I$, $h_i = \overline{\boldsymbol{o}}$ if $i \in J$ and $h_i = \overline{\boldsymbol{\iota}}$ if $i\notin J$. We call the collection $F(u)$ the collection of spin weighted components of $u$.
\end{prop}
\begin{proof}
The fact that the components are actually spin weighted functions with the claimed weight is a consequence of the $\C$-linearity of cospinors and the fact that $o(y \cdot z) = zo(y)$ and $\iota(y \cdot z) = z^{-1}\iota(y)$. The fact that $F$ is bijective follows from the construction of the inverse map. Indeed, if we fix $x\in \mathcal{M}$ and $y\in (\mathcal{A}_0)_x$, the spin weighted components evaluated at $y$ give exactly the image of a basis of $\otimes^{a}\mathcal{S}_x\otimes^{b}\mathcal{S}_x$ which correspond to the data of an element in $(\mathcal{S}_{a,b})_x$. The fact that this element does not depend on the choice of $y$ follows from the property \eqref{spinWeight}. 
\end{proof}

\begin{remark}
If $u$ has some regularity as a section, its spin weighted components have the same regularity (seen as section of the line bundle $\mathcal{B}(s,w)$) and reciprocally. It is the major advantage of spin weighted components of a smooth tensor: they are defined globally as smooth objects while being particularly adapted to the geometry.
\end{remark}

\begin{remark}
Thanks to the identification between $\mathcal{S}\otimes\overline{\mathcal{S}}$ and $T_\C\mathcal{M}$, we have also spin weighted components for tensor fields.
\end{remark}

\begin{remark}
Proposition \ref{correspondence} can also be understood if we consider $\boldsymbol{o}$ and $\boldsymbol{\iota}$ as spin weighted spinors (elements of $\Gamma(\mathcal{B}(s,w)\otimes \mathcal{S})$). Indeed, if $u\in \mathcal{S}_{a,b}$, spin weighted components are complete contractions of $u \otimes g_1 \otimes ... \otimes g_a \otimes h_1 \otimes ... \otimes h_b$ and are therefore sections of the tensor product of factors of the form $\mathcal{B}\left(\pm \frac{1}{2}, \pm \frac{1}{2}\right)$ and $\mathcal{B}\left(\pm \frac{1}{2}, \mp \frac{1}{2}\right)$ (and by remark \ref{productSpinWeighted} sections of some $\mathcal{B}(s,w)$).
\end{remark}

\begin{remark}
If $(o, \iota)$ is a local section of $\mathcal{A}_0$, it provide a local trivialization of $\mathcal{A}_0$ and therefore a local trivialization of $\mathcal{B}(s,w)$. The expression of a spin weighted component in this trivialization is obtained by replacing $\boldsymbol{o}$ by $o$, $\boldsymbol{\iota}$ by $\iota$, $\boldsymbol{l}$ by $l$, $\boldsymbol{m}$ by $m$ and $\boldsymbol{n}$ by $n$ in the expression of the component. Therefore the bold font notation is handy to take local trivializations. However, $\boldsymbol{o}$ do not depend on the choice of a particular local trivialization.
\end{remark}

\begin{remark}
\label{decompoMaxwellWeyl}
This decomposition is often used after a first decomposition of the cospinor or cotensor into symetric spinors (see \cite{penrose1984spinors} section 3.3 for more details about this type of decomposition). For example, because the electromagnetic tensor $\mathrm{F}$ is antisymmetric and real, it can be decomposed as: \begin{align}\mathrm{F} = \phi\otimes \overline{\epsilon} + \epsilon\otimes \overline{\phi} \label{eqF}
\end{align} (see (3.4.20) in \cite{penrose1984spinors} for details) where $\phi$ is a section of $(\mathcal{S}')^{\odot 2}$ (where $\odot$ is the symmetric product). Then, the spin weighted components of $\phi$ can be computed using equation \eqref{eqF}:
\begin{align*}
\phi(\boldsymbol{o}, \boldsymbol{o}) &= \mathrm{F}(\boldsymbol{l},\boldsymbol{m})\\
\phi(\boldsymbol{o}, \boldsymbol{\iota}) &= \frac{1}{2}\left(\mathrm{F}(\overline{\boldsymbol{m}},\boldsymbol{m})+\mathrm{F}(\boldsymbol{l},\boldsymbol{n})\right)\\
\phi(\boldsymbol{\iota}, \boldsymbol{\iota}) &= \mathrm{F}(\overline{\boldsymbol{m}},\boldsymbol{n})
\end{align*}
If we fix a local section $(o, \iota)$ of $\mathcal{A}_0$ and write the components in the associated local trivialization, we find the usual spin components of the electromagnetic tensor.

Similarly, we have the following decomposition for the Weyl tensor $W$ (see \cite{penrose1984spinors} 4.6.41):
\begin{align*}
W = \Psi\otimes \overline{\epsilon} \otimes \overline{\epsilon} + \epsilon\otimes\epsilon \otimes \overline{\Psi}
\end{align*}
where $\Psi$ is a section of $(\mathcal{S}')^{\odot 4}$. We can compute the spin weighted components of $\Psi$ from components of $W$:
\begin{align*}
\boldsymbol{\Psi}_0:=&\Psi(\boldsymbol{o}, \boldsymbol{o}, \boldsymbol{o}, \boldsymbol{o}) = W(\boldsymbol{l},\boldsymbol{m},\boldsymbol{l},\boldsymbol{m})\\
\boldsymbol{\Psi}_1:=&\Psi(\boldsymbol{o}, \boldsymbol{o}, \boldsymbol{o}, \boldsymbol{\iota}) = W(\boldsymbol{l}, \boldsymbol{m}, \boldsymbol{l}, \boldsymbol{n})\\
\boldsymbol{\Psi}_2:=&\Psi(\boldsymbol{o}, \boldsymbol{o}, \boldsymbol{\iota}, \boldsymbol{\iota}) = W(\boldsymbol{\overline{l}}, \boldsymbol{m}, \boldsymbol{m}, \boldsymbol{n})\\
\boldsymbol{\Psi}_3:=&\Psi(\boldsymbol{o}, \boldsymbol{\iota}, \boldsymbol{\iota}, \boldsymbol{\iota}) = W(\boldsymbol{l}, \boldsymbol{n}, \overline{\boldsymbol{m}}, \boldsymbol{n})\\
\boldsymbol{\Psi}_4:=&\Psi(\boldsymbol{\iota}, \boldsymbol{\iota},\boldsymbol{\iota}, \boldsymbol{\iota}) = W(\boldsymbol{\overline{m}}, \boldsymbol{n}, \boldsymbol{\overline{m}}, \boldsymbol{n})
\end{align*}
Since $\boldsymbol{l}$ and $\boldsymbol{n}$ (seen as maps from $\mathcal{A}_0$ to $T_\C \mathcal{M}$) are valued in the set of principal null vectors and the space-time is of type D, all the components vanish except $\boldsymbol{\Psi_2}$.
Note that there are different sign conventions for $\Psi_i$ (here we adopted the sign convention of \cite{penrose1984spinors} introduced in equations (4.11.9), but \cite{chandrasekhar1998mathematical} in chapter 1 equations (294) and \cite{teukolsky1973perturbations} in equation (1.3) adopt the opposite convention).
\end{remark}

\begin{remark}
\label{swCompSwValuedCospinor}
We can also define spin weighted components for sections of $\mathcal{S}_{a,b}\otimes \mathcal{B}(s,w)$ (using remark \ref{productSpinWeighted}).
\end{remark}

\subsection{Reduction of $\mathcal{A}_0$ and $\mathcal{N}_0$}
To simplify the computations it is interesting to find a smaller principal bundle with a representation such that the associated vector bundle is isomorphic to $\mathcal{B}(s,w)$.
We can consider $\mathcal{A}_{0,r}:= \mathcal{A}_0/\R_+^*$ (we quotient by the action of $\R_+^*\subset \C^*$). Similarly we define $\mathcal{N}_{0,r} := \mathcal{N}_0/\R_+^*$. We verify that the map $d$ induces a double cover between $\mathcal{A}_{0,r}$ and $\mathcal{N}_{0,r}$ (we still call this induced map $d$). Moreover $\mathcal{A}_{0,r}$ and $\mathcal{N}_{0,r}$ have both a structure of $U(1)$ principal bundle over $\mathcal{M}$. 
\begin{remark}
\label{identification}
According to proposition \ref{globalln}, we have a global choice of null independent principal real smooth vector fields $l$ and $n$. It enables us to make a global choice of representative for $\mathcal{A}_{0,r}$ and $\mathcal{N}_{0,r}$ (note that the ordering $(l,n)$ gives a choice of connected component). We have the following identifications for the reduced bundles 
\[ \mathcal{A}_{0,r}\simeq\left\{ (o,\iota)\in \mathcal{S}\times\mathcal{S}: o\otimes\overline{o} = l, \iota\otimes \overline{\iota} = n\right\}\]
\[ \mathcal{N}_{0,r}\simeq\left\{ m\in T_\C\mathcal{M} : g(m,m)=g(l,m)=g(n,m)=0, g(m,\overline{m})=-1 \text{ and }(n,l,m) \text{ is oriented}\right\}.\] Pay attention to the fact that the identification of $\mathcal{A}_{0,r}$ depends on the particular choice of $l$ and $n$ but the identification of $\mathcal{N}_{0,r}$ only depends on the ordering of $l$ and $n$.
\end{remark}
\begin{remark}
In both cases, the choice of $l$ and $n$ enables to associate each local smooth section of $\mathcal{A}_{0,r}$ (resp. $\mathcal{N}_{0,r}$) to a local smooth section of $\mathcal{A}_0$ (resp. $\mathcal{N}_0$). Therefore, when a choice of $l$ and $n$ has been made, we can work with $\mathcal{A}_{0,r}$ and $\mathcal{N}_{0,r}$ instead of $\mathcal{A}_0$ and $\mathcal{N}_0$.
\end{remark}

\section{Concrete computations in the subextremal Kerr exterior}
We define the Kerr metric with mass parameter $M$ and angular momentum per unit of mass $a$. We assume that $0<a<M$ (subextremal Kerr).
\[
g= \left(1-\frac{2Mr}{\rho^2}\right)\dd t^2 + \left(4Mar\frac{\sin^2\theta}{\rho^2}\right)\dd t \dd \phi - \frac{\rho^2}{\Delta_r}\dd r^2 - \rho^2 \dd \theta^2 - \sin^2\theta\left(r^2 + a^2 + 2Ma^2r\frac{\sin^2\theta}{\rho^2}\right)\dd \phi^2
\]
with
\begin{align*}
\Delta_r &:= r^2 - 2Mr + a^2\\
\rho^2 &:= r^2 + a^2 \cos^2\theta
\end{align*}
We define $r_0 := M+\sqrt{M^2-a^2}$ and we consider first the Kerr exterior $\mathcal{M}:= \R_t\times(r_0,+\infty)\times \mathbb{S}^2$. 
We also define the $Kerr_*$ coordinates: $(t_*, r, \theta_*, \phi_*) = (t + T(r), r, \theta, \phi + A(r))$ with $T(r):= \int_{r_1}^r \frac{a^2+r^2}{\Delta_r} \dd r$ and $A(r)=\int_{r_1}^r \frac{a}{\Delta_r}\dd r$ for some arbitrary (but fixed) $r_1\in (r_0, +\infty)$.
Kerr space time is an important example of Petrov type D space-time (Ricci-flat). In this section we explicit the previous definitions in this concrete case.

\subsection{Complete system of trivializations}
We now compute the concrete topology of the bundles in the Kerr case. We will show on the way that there is no global continuous oriented Newman Penrose tetrad (global continuous section of $\mathcal{N}_0$) nor global continuous normalized spin frame along the null directions (global section of $\mathcal{A}_0$).
Let $M>0$ and $a<M$. We endow $\mathcal{M}:=\R_t\times(r_0, +\infty)\times \mathbb{S}^2$ with the Kerr metric. It is a Petrov type D simply connected space-time.
We saw that in this case, $\mathcal{A}$ and $\mathcal{N}$ have two connected components (proposition \ref{simplyConnected}). Then, proposition \ref{globalln} tells us that a choice of a connected component is given by a choice of global smooth vector fields $(l,n)$ future oriented and independent along principal null directions. Here we take (Kinnersley's tetrad):
\begin{align}
\label{principalVectors}
l &= \frac{r^2+a^2}{\Delta_r}\partial_t + \partial_r + \frac{a}{\Delta_r}\partial_\phi \\
\label{principalVectors2}
n &= \frac{r^2 + a^2}{2\rho^2} \partial_t - \frac{\Delta_r}{2\rho^2} \partial_r + \frac{a}{2\rho^2}\partial_\phi 
\end{align}
We use the identification in remark \ref{identification} to describe $\mathcal{A}_{0,r}$ and $\mathcal{N}_{0,r}$. We define $p = r+ia\cos\theta$. Then we can see that
\begin{equation}
\label{vectorM}
m = \frac{ia\sin\theta}{\sqrt{2}p}\partial_t + \frac{1}{\sqrt{2}p}\partial_\theta + \frac{i}{\sqrt{2}p\sin\theta}\partial_{\phi}
\end{equation}
is a local section of $\mathcal{N}_{0,r}$ over $\mathcal{M}\setminus \R_t\times (r_0,+\infty)\times \left\{N, S\right\}$ where $N$ and $S$ are the north and south poles of $\mathbb{S}^2$. Note that the vector field $m$ cannot be extended to a smooth vector field on $\mathcal{M}$. However it provides a local trivialization of $\mathcal{N}_{0,r}$:
\[
\Psi_m:\begin{cases}
U(1)\times \R_t\times(r_0,+\infty)\times \left(\mathbb{S}^2\setminus{\left\{N,S\right\}}\right)  \rightarrow \mathcal{N}_{0,r} \\
(e^{i\rho}, x) \mapsto e^{i\rho}m(x)\\
\end{cases}
\]

\begin{remark}
Note that given a local trivialization $\Psi:U(1)\times U\rightarrow \mathcal{N}_{0,r} $ on $\mathcal{N}_{0,r}$, we can define a corresponding local trivialization on $\mathcal{N}_0$ by taking 
\[\tilde{\Psi}: \begin{cases}
\C^*\times U\rightarrow \mathcal{N}_0 \\
(z, x) \mapsto (|z|l(x), |z|^{-1}n(x), \Psi(\frac{z}{|z|},x))\\
\end{cases}
\]
\end{remark}

Our next goal is to write a complete system of local trivializations on $\mathcal{N}_{0,r}$.

We define the following map using stereographic coordinates relative to the north pole $(x_N, y_N)$ on $\mathbb{S}^2$
\[
\Psi_N:
\begin{cases}
U(1)\times \R_t \times(r_0, +\infty)\times \left(\mathbb{S}^2\setminus{\left\{N\right\}}\right) \rightarrow T_{\C}\mathcal{M}\\
(e^{i\rho}, t, r, x_N, y_N)\mapsto \frac{e^{i\rho}}{\sqrt{2}\left(r+ia\frac{(x_N^2+y_N^2)-1}{1+x_N^2+y_N^2}\right)}\left(\frac{2ia(x_N-iy_N)}{1+x_N^2+y_N^2}\partial_t+\frac{(x_N^2+y_N^2+1)}{2}\left(-\partial_{x_N}+i\partial_{y_N}\right)\right)
\end{cases}
\]
Note that $\Psi_N$ commutes with the projection $T_{\C}\mathcal{M}\rightarrow \mathcal{M}$.
We remark that on $U(1)\times \R_t \times(r_0, +\infty)\times \mathbb{S}^2\setminus{\left\{N,S\right\}}$ (where $U(1)$ is identified with $U(1)\times \left\{\text{Id}\right\}$) we have:
\[ \Psi_m = \frac{x_N+iy_N}{\sqrt{x_N^2+y_N^2}}\Psi_N = e^{i\phi}\Psi_N \]
where $\phi$ is the usual spherical coordinate on $\mathbb{S}^2$.
Using that $\mathcal{N}_{0,r}$ is closed in $T_\C\mathcal{M}$ and $\Psi_N$ is continuous, we deduce that $\Psi_N$ has values in $\mathcal{N}_{0,r}$. Moreover, since $\Psi_N$ is a smooth proper injective immersion as a function with values in $T_\C\mathcal{M}$, it remains true as a function with values in the submanifold $\mathcal{N}_{0,r}$. Since $U(1)\times\R_t \times(r_0, +\infty)\times \left(\mathbb{S}^2\setminus{\left\{N\right\}}\right)$ has the same dimension as $\mathcal{N}_{0,r}$, we deduce that $\Psi_N$ defines local coordinates on $\mathcal{N}_{0,r}$. Moreover, we check easily that it trivializes the action of $U(1)$.

We also introduce the map $\Psi_S$ using the stereographic coordinates relative to the south pole $(x_S,y_S)$:
\[
\Psi_S:
\begin{cases}
U(1)\times \R_t \times(r_0, +\infty)\times \left(\mathbb{S}^2\setminus{\left\{S\right\}}\right) \rightarrow T\mathcal{M}\\
(e^{i\rho}, t, r, x_S, y_S)\mapsto \frac{e^{i\rho}}{\sqrt{2}\left(r+ia\frac{1-(x_S^2+y_S^2)}{1+x_S^2+y_S^2}\right)}\left(\frac{2ia(x_S+iy_S)}{1+x_S^2+y_S^2}\partial_t+\frac{(x_S^2+y_S^2+1)}{2}\left(\partial_{x_S}+i\partial_{y_S}\right)\right)
\end{cases}.
\]
On $U(1)\times \R_t \times(r_0, +\infty)\times \left(\mathbb{S}^2\setminus{\left\{N,S\right\}}\right)$, we have:
\[\Psi_m = \frac{x_S-iy_S}{\sqrt{x_S^2+y_S^2}}\Psi_S = e^{-i\phi}\Psi_S. \]
As previously for $\Psi_N$, we deduce that $\Psi_S$ defines local coordinates on $\mathcal{N}_{0,r}$ and trivializes the action of $U(1)$. 

Note that (for $\omega \in \mathbb{S}^2\setminus{\left\{N,S\right\}}$), \[\Psi_N (e^{i\rho},t,r, \omega) = \frac{x_S-iy_S}{x_S+iy_S}\Psi_S(e^{i\rho}, t,r \omega) = \frac{x_N-iy_N}{x_N+iy_N}\Psi_S(e^{i\rho},t,r, \omega)\] and we deduce \[\Psi_S^{-1}\Psi_N(e^{i\rho},t,r,\omega) = \left(\frac{x_N-iy_N}{x_N+iy_N}e^{i\rho},t,r, \omega\right)\]

The complete system of local trivializations $(\Psi_N, \Psi_S)$ enables us to show easily the following proposition:
\begin{prop} There is no global continuous section of $\mathcal{N}_{0,r}$ (in other words, $\mathcal{N}_{0,r}$ is not the trivial bundle $U(1)\times\R_t\times(r_0,+\infty)\times\mathbb{S}^2$)
\end{prop}
\begin{proof}
We argue by contradiction. Let's assume the existence of a global continuous section $f$. Then we construct $f_1:=\text{pr}_{U(1)}\Psi_N^{-1}\circ f_{|_{\left\{0\right\}\times\left\{r_0+1\right\}\times \mathbb{S}^2\setminus{\left\{N\right\}}}}$ which is continuous ($\text{pr}_{U(1)}$ being the projection on $U(1)$). Using the stereographic coordinates relative to the north pole on $\left\{0\right\}\times\left\{r_0+1\right\}\times \mathbb{S}^2\setminus{\left\{N\right\}}$, we can see $f_1$ as a function from $\R^2$ to $U(1)$. Using the same construction with respect to the south pole ($f_2=\text{pr}_{U(1)}\Psi_S^{-1}\circ f_{|_{\left\{0\right\}\times\left\{r_0+1\right\}\times \mathbb{S}^2\setminus{\left\{S\right\}}}}$) and the identification using stereographic coordinates relative to the south pole, we obtain a continuous function $f_2$ from $\R^2$ to $U(1)$. The two constructions overlap and going through the various identifications, we get the relation $f_1(x,y) = \frac{(x+iy)^2}{x^2+y^2}f_2(\frac{x}{x^2+y^2}, \frac{y}{x^2+y^2})$ on $\R^2\setminus{0}$. We define $g:(0,+\infty)\times U(1) \rightarrow U(1)$ by $g(r, \omega):= f_1(r\omega)$. Since $f_1$ and $f_2$ are continuous at $(0,0)$, $g$ can be continuously extended by $g(0,\omega) = f_1(0)$ and $g(+\infty, \omega) = \omega^2f_2(0)$. This extension is a homotopy between two loops with different indices hence we have a contradiction and there is no global continuous section.
\end{proof}
\begin{remark}
A global continuous section on $\mathcal{N}_0$ composed with the projection of $\mathcal{N}_0$ onto $\mathcal{N}_{0,r}$ provides a global continuous section on $\mathcal{N}_{0,r}$. Therefore there is no global continuous section of $\mathcal{N}_0$ either. Similarly (using the map $d$) there is no global continuous section of $\mathcal{A}_0$ and $\mathcal{A}_{0,r}$.
\end{remark}

From the complete system of local trivialization on $\mathcal{N}_0$, we can deduce a complete system of local trivialization on $\mathcal{A}_0$. Indeed, the previous discussion shows that we have the sections $s_N = (l,n,e^{-i\phi}m)$ (smooth on $\R_t\times(r_0,+\infty)\times \mathbb{S}^2\setminus\left\{N\right\}$) and $s_S = (l,n,e^{i\phi}m)$ (smooth on $\R_t\times(r_0,+\infty)\times \mathbb{S}^2\setminus\left\{S\right\}$) of $\mathcal{N}_0$. Since $\R_t\times(r_0,+\infty)\times \mathbb{S}^2\setminus\left\{N\right\}$ is simply connected and $d$ is a double covering map, we have exaclty two lifts of $s_N$ as a local smooth section of $\mathcal{A}_0$ on $\R_t\times(r_0,+\infty)\times \mathbb{S}^2\setminus\left\{N\right\}$. We fix a choice $(o_N, \iota_N)$ of such a section. Then we define $(o, \iota)_S = (e^{i\phi}o_N, e^{-i\phi}\iota_N)$ which is smooth on $\R_t\times(r_0,+\infty)\times \mathbb{S}^2\setminus\left\{S,N\right\}$ and we want to prove that $(o,\iota)_S$ extends smoothly at $\R_t\times (r_0, +\infty)\times\mathbb{S}^2\setminus \left\{S\right\}$. To show that, we first check that $d\circ (o, \iota)_S = s_S$ on $\R_t\times(r_0,+\infty)\times \mathbb{S}^2\setminus\left\{S,N\right\}$. Let $x \in \R_t\times(r_0,+\infty)\times\mathbb{S}^2\setminus\left\{N,S\right\}$  we know that there exists a unique smooth lift $\tilde{s}_S$ of $s_S$ with value $(o,\iota)_S(x)$ at $x$. The set such that $\tilde{s}_S = (o,\iota)_S$ is open (we check that in an open set of trivialization containing a point $x_0$ such that $\tilde{s}_S(x_0) = (o,\iota)_S(x_0)$), closed (as a subset of $\R_t\times(r_0,+\infty)\times \mathbb{S}^2\setminus\left\{S,N\right\}$) by continuity of $\tilde{s}_S$ and $(o,\iota)_S$ and non empty. Therefore $(o,\iota)_S = \tilde{s}_S$. 
Finally, these two sections provide a complete system of trivializations of $\mathcal{A}_0$:
\begin{align*} A_N^{-1}: \begin{cases}\C^*\times \R_t\times (r_0, +\infty)\times \mathbb{S}^2\setminus \left\{N\right\}\rightarrow \mathcal{A}_0 \\
(z,x) \mapsto (zo_N(x),z^{-1}\iota_N(x))
\end{cases}\\
A_S^{-1}: \begin{cases}
\C^*\times \R_t\times (r_0,+\infty)\times \mathbb{S}^2\setminus\left\{S\right\} \rightarrow \mathcal{A}_0\\
(z,x) \mapsto (zo_S(x), z^{-1}\iota_S(x))
\end{cases}
\end{align*}
 with change of trivializations given by the map $A_N A_S^{-1}(z,x) = (e^{i\phi}z, x)$.
Note that there is no lift of the local section $(l,n,m)$ smooth on $\R_t\times (r_0, +\infty)\times \mathbb{S}^2\setminus\left\{N,S\right\}$. Indeed, the local section of $\mathcal{A}_0$ $(o_m,\iota_m):= (e^{i\frac{\phi}{2}}o_N, e^{-i\frac{\phi}{2}}\iota_N)$ defined on $\R_t\times (r_0, +\infty)\times \mathbb{S}^2\setminus\left\{\phi = 0\right\}$ is such a lift but it does not extend continuously to $\R_t\times (r_0, +\infty)\times \mathbb{S}^2\setminus\left\{N,S\right\}$. We denote by $A_m$ the local trivialization of $\mathcal{A}_0$ associated to $(o_m, \iota_m)$. It is traditionally used to write the Teukoslky operator.

\begin{remark}[Bundles on an extended Kerr space-time]
\label{extendedBundles}
We can also compute local trivializations of $\mathcal{A}_{0}$ and $\mathcal{N}_{0}$ on a larger Kerr space time. More precisely, using Kerr star coordinates $(t_*, r, \omega_*)$, we can extend the Kerr metric on a larger space-time which is given by $\mathcal{M}_{ext}:=\R_{t^*}\times (r_0 - \epsilon, +\infty)_{r}\times \mathbb{S}^2_{\omega_*}$. There are very few modifications with respect to the computations on the exterior. The main thing to note is that we cannot make the same choice of $l$ and $n$ as previously since they do not extend smoothly across the future horizon $\mathcal{H}=\left\{r = r_0\right\}$. Thus we renormalize them:
\begin{align*}
\tilde{l} &= \Delta_r l \\
\tilde{n} &= \Delta_r^{-1} n
\end{align*}
and extend them as principal independent future oriented principal null vector fields.
However the vector field $m$ extends smoothly to $\R_{t^*}\times (r_0 - \epsilon, +\infty)_{r}\times \left(\mathbb{S}^2_{\omega_*}\setminus \left\{N,S\right\}\right)$. Indeed we have in Kerr star coordinates:
\[
m = \frac{ia\sin\theta_*}{\sqrt{2}p}\partial_{t_*} + \frac{1}{\sqrt{2}p}\partial_{\theta_*} + \frac{i}{\sqrt{2}p\sin\theta_*}\partial_{\phi_*}
\]
where $p = r+ia\cos\theta_* = r+ia\cos\theta$ as previously. This expression defines an extension of $m$ on $\mathcal{M}_{ext}$ such that for all $x\in \mathcal{M}_{ext}$, $(\tilde{l}(x),\tilde{n}(x),m(x))\in \mathcal{N}_0$. As previously, we can compute explicitly a complete system of local trivializations of the bundles $\mathcal{N}_{0,r}$, $\mathcal{A}_{0,r}$, $\mathcal{N}_{0}$ and $\mathcal{A}_0$.
\end{remark}

\subsection{Link with the Hopf fibration}
In this section, we see $\mathbb{S}^3$ as the unit quaternions group that is to say \[\mathbb{S}^3 := \left\{ a+bi+cj+dk, (a,b,c,d)\in \R^4: a^2+b^2+c^2+d^2=1\right\}\subset \mathbb{H}.\] For $h\in \mathbb{S}^3$, the subset of imaginary quaternions $I=\left\{ bi+cj+dk: (b,c,d)\in \R^3\right\}$ is stable by  the map $c_h:h'\mapsto hh'h^*$ ($h^*$ is the conjugate of $h$ in the sense of quarternions and since $h \in \mathbb{S}^3$, $h^* = h^{-1}$). The map $c_h$ is even an orthogonal map for the usual norm on $I$ since $(hh'h^*)(hh'h^*)^* = hh'h'^*h^* = \left\|h'\right\|^2$. Therefore, $\Psi:h\mapsto c_h$ defines a Lie group morphism from $\mathbb{S}^3$ to $O(3)$ (note that $I$ is identified with $\R^3$ by sending $(i,j,k)$ to the canonical basis). Since $\mathbb{S}^3$ is connected and $c_1 = Id$, we have that $\Psi(\mathbb{S}^3) \subset SO(3)$. Finally, we compute the kernel of $\Psi$. Let $h$ be such that $c_h = Id$. Then for all $h'\in I$, we have:
\begin{align*}
hh'h^* = h'\\
hh' = h'h &\text{ (since $h^*h = 1$)}
\end{align*}
Therefore $h$ commute with every element of $I$. But $I+Z(\mathbb{H}) = \mathbb{H}$ (where $Z(\mathbb{H}) = \R$ is the center of $\mathbb{H}$. Therefore, $h\in Z(\mathbb{H})$. Finally, since $h\in \mathbb{S}^3$, we find $h = 1$ or $h=-1$.
Therefore, $\Psi$ is a Lie group morphism between two connected Lie groups with finite kernel of size two. Therefore it is a double covering map. We can identify $SO(3)$ with $\mathfrak{O}_{\mathbb{S}^2}$, the bundle of oriented orthonormal frame on $\mathbb{S}^2$ (we identify a matrix in $SO(3)$ with columns $C_1$, $C_2$ and $C_3$ with the basis $(C_2, C_3)\in T_{C_1}\mathbb{S}^2$). We denote by $\pi_{\mathbb{S}^2}: SO(3)\rightarrow \mathbb{S}^2$ the projection when $SO(3)$ is seen as the bundle of oriented orthonormal frame on $\mathbb{S}^2$ (therefore $\pi_{\mathbb{S}^2}(M)$ is the first column of the matrix with the identification that we have chosen).

The Hopf fibration can be defined as $H = \pi_{\mathbb{S}^2}\circ\Psi: \mathbb{S}^3\rightarrow \mathbb{S}^2$ (which is a smooth submersion). Let $u\in \mathbb{S}^2\subset I$, by definition $H^{-1}(u)= \left\{ h\in \mathbb{S}^3: hih^* = u\right\}$. We see that we have a right smooth fiber preserving action of $U(1) := \mathbb{S}^3\cap (\R+i\R)$ on $\mathbb{S}^3$ (given by right multiplication). Moreover, this action is simply transitive on each fiber since if $h_1, h_2 \in H^{-1}(u)$, then $g:=h_1^{-1}h_2$ is the only elements of $\mathbb{H}$ such that $h_1g = h_2$ and it belongs to $U(1)$ (indeed it commutes with $i$ and has norm 1). Therefore, the Hopf fibration is a principal $U(1)$-bundle. There is a unique action of $U(1)$ on $\mathfrak{O}_{\mathbb{S}^2}$ such that for all $x \in \mathbb{S}^3$ and $g\in U(1)$, $\Psi(x\cdot g) = \Psi(x)\cdot g^2$ and it is defined by $c_h\cdot g = c_{hg'} = c_{-hg'}$ for any $g'$ such that $g'^2= g$. Writing this more explicitly, we see that for a matrix $M\in SO(3)$ with columns $C_1, C_2, C_3$ , $ M = \Psi(h)$ for some $h \in \mathbb{S}^3$ with $\underset{(i,j,k)}{\text{Mat}}hih^* = C_1$, $\underset{(i,j,k)}{\text{Mat}}hjh^* = C_2$ and $\underset{(i,j,k)}{\text{Mat}}hkh^*= C_3$. Then if $g = \cos(\rho) +i\sin(\rho) \in U(1)$ and $g' = \cos\left(\frac{\rho}{2}\right)+i\sin\left(\frac{\rho}{2}\right)$, $M\cdot g$ has columns 
\begin{align*}\underset{(i,j,k)}{\text{Mat}} hg' i (hg')^* &= C_1\\
\underset{(i,j,k)}{\text{Mat}} hg' j (hg')^* &= \underset{(i,j,k)}{\text{Mat}}\left((\cos^2\left(\frac{\rho}{2}\right)-\sin^2\left(\frac{\rho}{2}\right)) hjh^* + 2\sin\left(\frac{\rho}{2}\right)\cos\left(\frac{\rho}{2}\right) hkh^*\right)\\
&= \cos(\rho)C_2 + \sin(\rho)C_3\\
\underset{(i,j,k)}{\text{Mat}} hg'k(hg')^* &= -\sin(\rho)C_2 + \cos(\rho)C_3. 
\end{align*}
The map $Id_{\R_t\times(r_0,+\infty)}\times H$ enables to put a structure of $U(1)$ principal bundle on $\R_t\times(r_0,+\infty)\times \mathbb{S}^3$ and similarly, we put a structure of $U(1)$ principal bundle on $\R_t\times(r_0,+\infty)\times \mathfrak{O}_{\mathbb{S}^2}$.

\begin{prop}
We define the map:
\[
f:
\begin{cases}
\R_t\times(r_0, +\infty)\times\mathfrak{O}_{\mathbb{S}^2}\rightarrow T_\C\mathcal{M} \\
(t,r,(X,Y)\in (\mathfrak{O}_{\mathbb{S}^2})_\omega) \mapsto -\frac{ia \left<X+iY, e_3\right>_{\R^3}}{\sqrt{2}p}\partial_t+\frac{1}{\sqrt2 p}X+\frac{i}{\sqrt{2}p}Y\in T_{(t,r,\omega)}\mathcal{M}
\end{cases}
\]
where elements of $\mathbb{S}^2$ (resp. $T\mathbb{S}^2$) are represented by unit vectors (resp. triple of unit vectors) in $\R^3$ (and $e_1, e_2, e_3$ is the canonical basis of $\R^3$) and $p = r+i\cos(\theta)$. The notation $\left< . \right>_{\R^3}$ denotes the canonical scalar product on $\R^3$ extended to a $\C$-bilinear form on $\C^3$ (therefore it is not hermitian).

We have $f:\R_t\times(r_0, +\infty)\times\mathfrak{O}_{\mathbb{S}^2}\rightarrow \mathcal{N}_{0,r}$ and it is an isomorphism of principal bundle.
\end{prop}
\begin{proof}
Let $(t,r,\omega)\in \mathcal{M}$ and $(X,Y)\in \left(\mathfrak{O}_{\mathbb{S}^2}\right)_\omega$. To prove that $f$ has values in $\mathcal{N}_{0,r}$, we have to prove that $m:=-\frac{ia \left<X+iY, e_3\right>_{\R^3}}{\sqrt{2}p}\partial_t+\frac{1}{\sqrt2 p}X+\frac{i}{\sqrt{2}p}Y$ is null, orthogonal to any principal null vector and $g(m,\overline{m}) = -1$.

We can write the Kerr metric on the form:
\begin{align*}
g = \left(1-\frac{2Mr}{\rho^2}\right)\dd t^2 + \frac{4Mar}{\rho^2}\dd t (\sin^2\theta\dd \phi) - \frac{\rho^2}{\Delta_r}\dd r^2 - \rho^2 g_{\mathbb{S}^2} - a^2\left(1+\frac{2Mr}{\rho^2}\right)(\sin^2\theta\dd \phi)^2
\end{align*}
Moreover, if we see elements $\omega\in\mathbb{S}^2$ as units vectors $\begin{pmatrix}\omega_x \\ \omega_y \\ \omega_z\end{pmatrix}$ on $\R^3$, we have:
\begin{align*} \sin^2\theta\dd \phi = -\omega_y\dd x + \omega_x\dd y \\
\end{align*}
We can therefore compute:
\begin{align*}
2p^2g(m,m) = -a^2\left( \left<X+iY, e_3\right>^2+\left<X+iY, -\omega_y e_1 + \omega_x e_2\right>^2-\frac{2Mr}{\rho^2}\left(\left<X+iY, e_3\right>-i\left<X+iY, -\omega_ye_1+\omega_x e_2\right>\right)^2\right)
\end{align*}
But note that $-\omega_y e_1 + \omega_x e_2 = e_3\times\omega$, $\omega\times X = Y$ and $Y\times \omega = X$ (since $\omega, X, Y$ is a direct orthonormal basis of $\R^3$). Then, by definition of the cross product on $\R^3$, 
\begin{align*}
\left< X, e_3\times\omega \right> =& \det( X,e_3,\omega) \\
=& \det(e_3,\omega, X)\\ 
=& \left< \omega\times X, e_3\right>  \\
=& \left< Y, e_3\right>\\
\left< Y, e_3\times\omega\right> =& -\left<X, e_3\right>
\end{align*} 
Therefore, $\left<X+iY, e_3\times\omega\right> = \left<Y-iX, e_3\right> = -i\left<X+iY, e_3\right>$.
We deduce that $2p^2g(m,m) = 0$.

We compute also
\begin{align*}
g(\sqrt{2}p m,\Delta_r l) =& \left(1-\frac{2Mr}{\rho^2}\right)\left(-ia\left< X+iY, e_3\right>\right)(r^2+a^2) + \frac{2Ma^2r}{\rho^2}\sin^2\theta(-ia\left<X+iY, e_3\right>)\\
& + \frac{2Mar(r^2+a^2)}{\rho^2}\left<X+iY, e_3\times \omega\right>-a\rho^2\left<X+iY, e_3\times \omega\right>\\
&-a^3\left(1+\frac{2Mr}{\rho^2}\right)\left<X+iY, e_3\times \omega\right>\sin^2\theta\\
=& 0 \quad\text{(using $\left<X+iY, e_3\times\omega\right> = -i\left<X+iY, e_3\right>$)}
\end{align*}
Very similar computations show:
\begin{align*}
g(\sqrt{2}p m , \Delta_r^{-1} n) = 0\\
g(m,\overline{m}) = 1
\end{align*}
The previous computations show that $f$ has values in $\mathcal{N}_{0,r}$.
Moreover $f$ is smooth and so is its inverse:
\begin{align*}
f^{-1}: m\in (\mathcal{N}_{0,r})_{(t,r,\omega)}\mapsto (t,r, \sqrt{2}p(\Re(\text{pr}_{T_\omega\mathbb{S}^2, \partial_t}m), \Im(\text{pr}_{T_\omega\mathbb{S}^2, \partial_t}m)))
\end{align*}
where $\text{pr}_{\mathbb{S}^2, \partial_t}$ is the linear projection on $T_\omega\mathbb{S}^2$ parallel to $\partial_t$.
Therefore, it is a diffeomorphism. The compatibility with projections maps is immediate. It remains to prove that $f$ is compatible with the actions of $U(1)$. Let $e^{i\rho} \in U(1)$.
\begin{align*}
(\omega, X,Y)\cdot{e^{i\rho}} &= (\omega, \cos(\rho)X + \sin(\rho)Y, -\sin(\rho)Y+\cos(\rho)X)\\
&= (\omega, \Re(e^{i\rho} (X+iY)), \Im(e^{i\rho} (X+iY)))
\end{align*}
Therefore, if $(X,Y) \in T_{\omega}\mathbb{S}^2$:
\begin{align*}
f(t,r,(X,Y)\cdot{e^{i\rho}}) &= -\frac{ia \left<e^{i\rho}(X+iY), e_3\right>_{\R^3}}{\sqrt{2}p}\partial_t + e^{i\rho}(X+iY)\\
&= e^{i\rho}f(t,r,(X,Y))\\
&= f(t,r,(X,Y))\cdot e^{i\rho}
\end{align*}
\end{proof}

The double cover map $\tilde{d}:= f\circ\Psi$ satisfies $\tilde{d}(x\cdot g) = \tilde{d}(x)\cdot g^2$ for all $g\in U(1)$ and $x\in\R_t\times (r_0,+\infty) \times \mathbb{S}^3$ and the following diagram is commutative
\[
\begin{tikzcd}
\R_t\times (r_0,+\infty) \times \mathbb{S}^3 \arrow[d, "Id\times H"] \arrow[r, "\tilde{d}"] & \mathcal{N}_{0,r} \arrow[ld, "\pi_{\mathcal{N}_{0,r}}"]\\
\mathcal{M}
\end{tikzcd}
\] 

We see that $\tilde{d}:\R_t \times (r_0, +\infty) \times \mathbb{S}^3 \rightarrow \mathcal{N}_{0,r}$ is very similar to $d: \mathcal{A}_{0,r} \rightarrow \mathcal{N}_{0,r}$. Indeed, the two are isomorphic as we see in the following proposition:
\begin{prop}
Given $u \in \R_t \times (r_0, +\infty) \times \mathbb{S}^3$ and $v\in \mathcal{A}_{0,r}$ such that $\tilde{d}(u) = d(v)$, we have a unique isomorphism of principal bundles $\mathfrak{G}: \R_t \times (r_0, +\infty) \times \mathbb{S}^3 \rightarrow \mathcal{A}_{0,r}$ such that the following diagram is commutative:
\[
\begin{tikzcd}
\R_t \times (r_0, +\infty) \times \mathbb{S}^3 \arrow{rd}{\tilde{d}} \arrow{rr}{\mathfrak{G}} & & \mathcal{A}_{0,r} \arrow{ld}{d}\\
& \mathcal{N}_{0,r} &
\end{tikzcd}
\]
and such that $\mathfrak{G}(u) = v$.
\end{prop}

\begin{proof}
Since $\R_t \times (r_0, +\infty) \times \mathbb{S}^3$ is simply connected and $d$ is a covering map, $\tilde{d}$ admits a unique lift through $d$ to a smooth map $\mathfrak{G}: \R_t \times (r_0, +\infty) \times \mathbb{S}^3\rightarrow \mathcal{A}_{0,r}$ such that $\mathfrak{G}(u) = v$. A priori $\mathfrak{G}$ is only a smooth map. It remains to show that $\mathfrak{G}$ is in fact an isomorphism of principal bundles. First note that for all $y\in \R_t \times (r_0, +\infty) \times \mathbb{S}^3$, we have $(Id\times H)(y) = \pi_{\mathcal{A}_{0,r}}(\mathfrak{G}(y))$ (using the commutative diagrams).

Let $V$ be a small open subset of $\mathcal{M}$ and $y: V\rightarrow \R_t \times (r_0, +\infty) \times \mathbb{S}^3$ be a local section of $\R_t \times (r_0, +\infty) \times \mathbb{S}^3$. Then $\mathfrak{G}\circ y$ is a local section of $\mathcal{A}_{0,r}$ and $\tilde{d}\circ y$ is a local section of $\mathcal{N}_{0,r}$.  Let $\Psi_1$ be the local trivialization of $\R_t\times(r_0,+\infty)\times\mathbb{S}^3$ such that $\Psi_1(x, e^{i\rho}) = e^{i\rho}\cdot y(x)$, $\Psi_2$ the local trivialization of $\mathcal{A}_{0,r}$ such that $\Psi_2(x, e^{i\rho}) = e^{i\rho}\cdot\mathfrak{G}(y(x))$ and $\Phi$ the local trivialization of $\mathcal{N}_{0,r}$ such that $\Phi(x, e^{i\rho}) = e^{i\rho}\cdot \tilde{d}(y(x))$. Then for $x \in V$ and $a\in U(1)$, $\Psi_2^{-1}\circ\mathfrak{G} \circ \Psi_1(x,a) = (x, \gamma(x,a))$ where $\gamma$ is the unique continuous $U(1)$-valued function such that $\gamma(x,a)\cdot \mathfrak{G}(y(x))= \mathfrak{G}(a \cdot y(x))$. In particular $\gamma(x, 1) = 1$. Moreover, $\Phi^{-1}\circ \tilde{d} \circ \Psi_1 (x, a) = \Phi^{-1}\circ d \circ \Psi_2(x,a) = (x, a^2)$. For $x\in V$ and $a\in U(1)$:
\begin{align*}
\Phi^{-1}\circ d \circ \mathfrak{G}\circ \Psi_1 (x, a) &= \Phi^{-1}\circ d \circ \Psi_2\circ \Psi_2^{-1}\circ \mathfrak{G}\circ \Psi_1 (x,a)\\
&= (x, \gamma(x,a)^2)\\
\end{align*} 
And we also have
\begin{align*}
\Phi^{-1}\circ d \circ \mathfrak{G} \circ \Psi_1 (x, a) &= \Phi^{-1} \circ \tilde{d}\circ \Psi_1(x,a)\\
 &= (x, a^2)
\end{align*}
We deduce $\gamma(x,a)^2 = a^2$.

Then the following diagram is commutative:
\[
\begin{tikzcd}
& U(1) \arrow{d}{a\mapsto a^2} \\
V\times U(1) \arrow{ur}{\gamma} \arrow{r}{(x,a)\mapsto a^2} & U(1) 
\end{tikzcd}
\]
Because $a\mapsto a^2$ is a covering map, we have the uniqueness of such a continuous lift with $\gamma(x,1)=1$. We deduce $\gamma = \text{pr}_{U(1)}$. This proves that $\mathfrak{G}$ is an isomorphism of principal bundles.

\end{proof}
\begin{remark}
The previous proposition shows that there are exactly two choices $\mathfrak{G}_1$ and $\mathfrak{G}_2$ for the isomorphism and we have for all $a \in \R_t \times (r_0, +\infty) \times \mathbb{S}^3$, $\mathfrak{G}_1(a) = \mathfrak{G}_2(a\cdot(-1))$. We choose one of the two and call it $\mathfrak{G}$.
\end{remark}
Thanks to the previous proposition, we have now a concrete description of $\mathcal{A}_{0,r}$ and we can use it to define spin weighted functions as in \cite{dafermos2019boundedness} (section 2.2). The concrete description avoids the reference to spin frames.

\subsection{Stationarity}
In this section, we introduce the notion of a trivial (vector or principal) bundle with respect to a factor in a product decomposition and we apply this notion to the bundle $\mathcal{B}(s,s)$ in the Kerr case.
We consider a manifold $\mathcal{M}$ with a product decomposition $\Psi: \mathcal{M} \rightarrow \mathcal{X}\times \mathcal{Y}$ ($\Psi$ being a fixed diffeomorphism).
\begin{definition}
We say that a (vector or principal) bundle $p_E:E\rightarrow \mathcal{M}$ is trivial with respect to $\mathcal{X}$ in the decomposition given by $\Psi$ if there exists an isomorphism $f$ of (vector or principal) bundle over $\Psi$ between $E$ and the bundle $\mathcal{X}\times F$ (by definition it is the product of the trivial bundle $Id:\mathcal{X}\rightarrow \mathcal{X}$ and some bundle $p_{F}:F\rightarrow \mathcal{Y}$). In particular, we have the following commutative diagram:
\[
\begin{tikzcd}
E \arrow{d}{p_E} \arrow{r}{f} & \mathcal{X}\times F\arrow{d}{Id_{\mathcal{X}}\times p_F}\\
\mathcal{M} \arrow{r}{\Psi} & \mathcal{X}\times\mathcal{Y}
\end{tikzcd}
\]
We will say that $f$ is a semi-trivialization of the bundle.
\end{definition}
\begin{remark}
It is equivalent to say that $E$ is (isomorphic to) the pullback of a bundle $F$ on $\mathcal{Y}$ by the second projection (indeed, this pullback bundle is exactly the bundle $\mathcal{X}\times F$).
\end{remark}
\begin{remark}
In the following we fix the identification between $\mathcal{M}$ and $\Psi$. Therefore we will assume $\mathcal{M} = \mathcal{X}\times\mathcal{Y}$.
\end{remark}
\begin{remark}
\label{changeSemiTriv}
If $f: E\rightarrow \mathcal{X}\times F$ and $f': E\rightarrow \mathcal{X}\times F'$ are two semi-trivializations, we have $f\circ f'^{-1}(x,z) = (x,\gamma(z))$ where $\gamma: F'\rightarrow F$ is an isomorphism of vector (or principal) bundles.
\end{remark}

\begin{prop}
\label{contractileTrivial}
Let $E$ be a finite rank vector bundle over $\mathcal{M} = \mathcal{X}\times\mathcal{Y}$ ($\mathcal{M}$ is paracompact since it is a smooth manifold).
If $\mathcal{X}$ is contractile, then $E$ is trivial with respect to $\mathcal{X}$.
\end{prop}
\begin{proof}
Let $f: [0,1]\times \mathcal{X}\times\mathcal{Y}\rightarrow \mathcal{X}\times \mathcal{Y}$ be a smooth map such that $f_0 = Id_{\mathcal{X}\times\mathcal{Y}}$ and $f_1(x,y)= (x_0,y)$ for all $(x,y)\in \mathcal{X}\times \mathcal{Y}$ ($x_0$ is some element of $\mathcal{X}$). We define $\mathcal{E} = f^*E$. Then, for example by proposition 1.7 in \cite{hatcher2003vector} (rather its direct equivalent in the smooth case, obtained by minor modification in the proof), $\mathcal{E}_{|_{\left\{0\right\}\times\mathcal{X}\times\mathcal{Y}}}$ is isomorphic to $\mathcal{E}_{|_{\left\{1\right\}\times\mathcal{X}\times\mathcal{Y}}}$
\end{proof}

\begin{prop}
\label{statSyst}
Let $E$ be a vector bundle over $\mathcal{X}\times\mathcal{Y}$.
If $\mathcal{C}=(U_\alpha, \Psi_\alpha)$ is a complete system of local trivializations for $E$ such that all the transition maps $g_{\alpha, \alpha'}\in C^\infty(U_\alpha\cap U_{\alpha'}; Gl_n(\R))$ are independent of the first factor (ie factorize as $g_{\alpha,\alpha'} = \tilde{g}_{\alpha,\alpha'}\circ \pi_2$ where $\pi_2$ is the second projection on $\mathcal{X}\times \mathcal{Y}$), then there exists a semi-trivialization $f_\mathcal{C}: E\rightarrow \mathcal{X}\times F_\mathcal{C}$ where $F_{\mathcal{C}}$ is the vector bundle on $\mathcal{Y}$ given by the transition maps $\tilde{g}_{\alpha,\alpha'}$.
\end{prop}
\begin{remark}
We will say that the system $\mathcal{C}$ is stationary with respect to $\mathcal{X}$. As mentioned in the proof of the proposition, a stationary system of trivialization comes naturally with a semi-trivialization of $E$.
\end{remark}
\begin{proof}
We first define the vector bundle $F_{\mathcal{C}}$ by $\coprod_{\alpha} \pi_2(U_\alpha)\times \R^n / (\tilde{g}_{\alpha,\alpha'})$ (the family $\tilde{g}_{\alpha,\alpha'}$ satisfies the cocycle condition). We denote by $i_\alpha$ the natural maps from $U_\alpha\times \R^n$ to $F_{\mathcal{C}}$ (composition of the injection in the disjoint union and the projection in the quotient). These maps are continuous injective and by definition of the quotient, on $U_\alpha\cap U_\alpha'$, $i_\alpha^{-1} i_{\alpha'}(y,v) = (y, g_{\alpha,\alpha'}(y)v)$. Then we define the map $f_\mathcal{C}$ by $f_{\mathcal{C}}(z) = (\Psi_\alpha(z)_1, i_\alpha(\Psi_\alpha(z)_2,\Psi_\alpha(z)_3))$ if $z\in U_\alpha$ (the index 1,2 and 3 refers to the components in the product decomposition). This does not depend on the choice of $\alpha$ such that $z\in U_\alpha$ and $(Id_\mathcal{X}\times i_{\alpha}^{-1}) \circ f_\mathcal{C}\circ \Psi_\alpha^{-1}$ is the identity of $U_\alpha\times \R^n$. Therefore, $f_\mathcal{C}$ is an isomorphism of vector bundles.
\end{proof}
\begin{prop}
Reciprocally, if we have a semi-trivialization $f: E\rightarrow \mathcal{X}\times F$ and a complete system of local trivializations $(U_\alpha, \Psi_\alpha)$ on $F$, we can define a complete system of local trivialization on $E$ with transition maps independent of the first factor.
\end{prop}
\begin{proof}
We just take the system $(\mathcal{X}\times U_\alpha, (Id_{\mathcal{X}}\times \Psi_\alpha) \circ f)$.
\end{proof}

We now apply these notions to the study of $\mathcal{B}(s,w)$. 
First note that the complete system of trivializations $(A_N, A_S)$ of $\mathcal{A}_0$ has a transition map depending only on $\phi$. As a consequence, the associated sytem of trivializations of $\mathcal{B}(s,w)$ has also a transition map depending only on $\phi$ (see remark \ref{changeOfMap}). Therefore, we are in the context of proposition \ref{statSyst} and $\mathcal{B}(s,w)$ is trivial with respect to the factor $\R_t \times (r_0, +\infty)$ (we could also have used proposition \ref{contractileTrivial} but proposition \ref{statSyst} provides a concrete semi-trivialization $f:\mathcal{A}_0 \rightarrow \R_t \times (r_0, +\infty) \times \mathcal{B}_{\mathbb{S}^2}(s,w)$ associated with the stationary system of local trivializations). The map $f$ enables to identify $\Gamma(\mathcal{B}(s,w))$ with $C^{\infty}(\R_t\times (r_0,+\infty), \Gamma(\mathcal{B}_{\mathbb{S}^2}(s,w)))$.

\section{Connections on the bundles and GHP formalism}
%We now define the differential operators of the GHP formalism and use them to rewrite the Teukolsky operator. 
There are different methods to define the spin connection and the GHP connection, for example we could write them down explicitly in an expression involving coefficients depending only on the Levi-Civita connection and the metric (see remark \ref{computationNabla} and proposition \ref{computationGHPConnection}) and check that it defines a linear connection. For computational purpose, these definitions are enough but they are not insightful. The spin connection is the unique connection $\nabla$ such that $j^*\nabla_{LC} = \nabla\otimes\nabla$ (where $\nabla_{LC}$ is the Levi-Civita connection) and $\nabla\epsilon = 0$ and these properties can be used as an alternative definition as well. However, here we choose an other definition. We start with the Levi-Civita connection and move it naturally through the different bundles involved. The advantage of this method is that each step is very natural, moreover it gives a good understanding of where the GHP connection comes from. Its main drawback is that it is a little longer than the direct definitions and involves some elementary knowledge about principal connections.

Therefore, we give a brief reminder about principal connections in order to have a self contained presentation. A more detailed introduction to this topic can be found in \cite{kobayashi1963foundations} (second chapter).
\subsection{Principal connection}
\begin{definition}
Let $\pi_E: E\rightarrow \mathcal{M}$ be a principal bundle with structure group $G$ (multiplicative, with neutral element denoted by 1). Let $\mathfrak{g}$ be the Lie algebra of $G$.
A principal connection $\omega$ is a $\mathfrak{g}$-valued one form on $E$ such that:
\begin{itemize}
\item For all $g\in G$, $Ad_g R_g^* \omega = \omega$ where $Ad_g: \mathfrak{g}\rightarrow \mathfrak{g}$ is the adjoint representation ($Ad_g(\xi):=\frac{\dd}{\dd t}_{|t=0}g\exp(t\xi)g^{-1}$) and $R_g: E \rightarrow E$ is the right action of $g$ on $E$.
\item For all $\xi \in \mathfrak{g}$ and all $e\in E$, $\omega_e(\dd_{1} i_e \xi) = \xi$ where we have used the map
\[
i_e:\begin{cases}
G \rightarrow E \\
g \mapsto e\cdot g
\end{cases}
\]
\end{itemize}
\end{definition}
\begin{remark}
\label{egaliteEns}
Differentiating the relation $\pi_E\circ i_a(g) = a$, we have the inclusion $\dd_1 i_a(\mathfrak{g}) \subset ker (\dd_a \pi_E)$. Moreover, $\dim(ker(\dd_a \pi_E)) = \dim(\pi_E^{-1}(\left\{a\right\}) = \dim(G) = \dim(\mathfrak{g}) = \dim \dd_1 i_a(\mathfrak{g})$. Therefore $\dd_1 i_a(\mathfrak{g}) = ker (\dd_a \pi_E)$.
\end{remark}
\begin{lemma}
\label{decompoTangent}
For all $e\in E$,
$ker(\omega_e)\oplus ker(\dd_e \pi_E) = T_eE$. In particular $\dd_e {\pi_E}_{|ker(\omega_e)}$ is an isomorphism between $ker(\omega_e)$ and $T_{\pi_E(e)}\mathcal{M}$.
\end{lemma}
\begin{proof}
The image of the injective map $\dd_1 i_e$ is exactly $ker(\dd_e \pi_E)$ (see remark \ref{egaliteEns}). So the second point in the definition of $\omega$ implies that $ker(\omega_e)\cap ker(\dd_e \pi_E) = \left\{0\right\}$ and $\text{dim} Ran(\omega_e) = \text{dim} \mathfrak{g} = \text{dim}(ker(\dd_e \pi_E))$. Moreover, $\text{dim}(ker(\omega_e)) = \text{dim}(E)-\text{dim}(Ran(\omega_e))$ so we deduce that $\text{dim}(ker(\omega_e)) + \text{dim}ker(\dd_e \pi_E) = \text{dim}(E)$ and we have the lemma.
\end{proof}

From now, for $a\in E$, we denote by $H_a := ker(\omega_a)$. The following property gives a way to construct a linear connection (on an associated vector bundle) from a principal connection.
\begin{prop}
\label{principalToLinear}
Let $\rho: G\rightarrow GL(V)$ be a representation of $G$ and let $\mathcal{F}$ be the vector bundle associated to this representation. Then, a smooth section $s$ of $\mathcal{F}$ is naturally identified with a smooth function $f:E\rightarrow V$ such that $f(x\cdot g) = \rho(g^{-1})(f(x))$ and we can define (for $X\in T_x\mathcal{M}$) $(\nabla_X s)(x):= [a,\dd_{a}f((\dd_a {\pi_{E}}_{|_{H_a}})^{-1}(X))]$ where $a$ is any element of $\pi_{E}^{-1}(\left\{x\right\})$. With this definition, $\nabla$ is a linear connection on $\mathcal{F}$.
\end{prop}
\begin{proof}
The definition does not depend on the choice of $a$. If we choose $a'= a\cdot g$, we have to prove that $[a,\dd_{a}f((\dd_a {\pi_{E}}_{|_{H_a}})^{-1}(X))] = [a',\dd_{a'}f((\dd_{a'} {\pi_{E}}_{|_{H_{a'}}})^{-1}(X))]$. We have
\[[a', \dd_{a'}f((\dd_{a'} {\pi_{E}}_{|_{H_{a'}}})^{-1}(X))] = [a, \rho(g)\dd_{a'}f((\dd_{a'} {\pi_{E}}_{|_{H_{a'}}})^{-1}(X))]\] so it remains to prove that \[\rho(g)\dd_{a'}f((\dd_{a'} {\pi_{E}}_{|_{H_{a'}}})^{-1}(X)) = \dd_{a}f((\dd_a {\pi_{E}}_{|_{H_a}})^{-1}(X))\]. We have the following facts
\begin{itemize}
\item $\dd_{a\cdot g} \pi_E \dd_a R_g = \dd_a \pi_E$ (by differentiating $\pi_E\circ R_g = \pi_E$)
\item $\dd_{a\cdot g} {\pi_E}_{|_{H_{a\cdot g}}} (\dd_a R_g)_{|_{H_{a}}} = \dd_a {\pi_E}_{|_{H_a}}$ using the previous point and the fact that $R_g(H_a)=H_{a\cdot g}$.
\item $\dd_{a\cdot g} {\pi_E}_{|_{H_{a\cdot g}}}^{-1} = \dd_a R_g \dd_a {\pi_E}_{|_{H_a}}^{-1}$ (using the previous point)
\item $\dd_{a\cdot g} f \dd_a R_g= \rho(g^{-1})\dd_a f$ (by differentiating $f(a\cdot g) = \rho(g^{-1})f(a)$)
\end{itemize}
We use that to conclude $\rho(g)\dd_{a'}f((\dd_{a'} {\pi_{E}}_{|_{H_{a'}}})^{-1}(X)) = \dd_{a}f((\dd_a {\pi_{E}}_{|_{H_a}})^{-1}(X))$. Now we have to prove that $\nabla$ is a linear connection. We obviously have $\nabla_{\lambda X + Y} = \lambda\nabla_X+\nabla_Y$. Let $h$ be a smooth function on $\mathcal{M}$. The section $hs$ is associated with the function $\tilde{f} = (h\circ \pi_{E}) f$ and $\dd_a \tilde{f} = f(a)\dd_{\pi_{E}(a)}h\dd_a \pi_{E} + h(\pi_E(a))\dd_a f$. Therefore, we have \[\dd_{a}\tilde{f}((\dd_a {\pi_{E}}_{|_{H_a}})^{-1}(X)) = \dd_{\pi_{E}(a)}h(X) f+ h(\pi_{E}(a))\dd_{a}f((\dd_a {\pi_{E}}_{|_{H_a}})^{-1}(X)).\] We deduce that $\nabla$ defines a linear connection on $\mathcal{F}$.
\end{proof}

A useful lemma to compute the such defined connection is the following:
\begin{lemma}
\label{computationLinear}
We use the notation of proposition \ref{principalToLinear}. Let $e$ be a local smooth section of $E$ around some $x_0\in \mathcal{M}$ such that $\omega \circ \dd_{x_0} e = 0$. Let $X \in T_{x_0}\mathcal{M}$ and let $s$ be a local smooth section of $\mathcal{F}$ on some open neighborhood $U$ of $x_0$ such that $s(x) = [(e(x), v(x))]$ with $v: U\rightarrow V$ smooth. Then, $\nabla_X s = [(e(x_0), \dd_{x_0} v(X))]$
\end{lemma}
\begin{proof}
The equivariant function associated to $s$ is $f: \pi_E^{-1}(U) \rightarrow V$ such that for all $x\in U$, $f(e(x)) = v(x)$. Let $X\in T_{x_0}\mathcal{M}$. We use the definition of the connection $\nabla$ to write
\begin{align*}
\nabla_X s = [(e(x_0), \dd_{e(x_0)}f \left(\dd_{e(x_0)}{\pi_E}_{|_{H_{e(x_0)}}}\right)^{-1}(X))]
\end{align*}
Moreover, we have $\dd_{x_0}e (X) \in ker(\omega_{e(x_0)}) = H_{e(x_0)}$ and $\dd_{e(x_0)}\pi_E \dd_{x_0} e(X) = \dd_{x_0}(\pi_E\circ e)(X) = X$. We deduce that $\left(\dd_{e(x_0)}{\pi_E}_{|H_{e(x_0)}}\right)^{-1}(X) = \dd_{x_0}e (X)$. As a consequence
\begin{align*}
\nabla_X s &= [(e(x_0), \dd_{e(x_0)}f \dd_{x_0}e (X))]\\
&= [(e(x_0), \dd_{x_0} (f\circ e)(X))]\\
&= [(e(x_0), \dd_{x_0} v(X))]
\end{align*}
\end{proof}
\begin{remark}
In the case where $E$ is the frame bundle of $\mathcal{F}$, the condition $\omega \circ \dd_{x_0} e = 0$ amounts to say that the derivative of the local frame vanishes at $x_0$. Then the previous lemma tells us that in this local frame, we can compute covariant derivatives of a section of $\mathcal{F}$ by taking usual derivatives of the coordinates.
\end{remark}
\begin{remark}
\label{sectionPara}
Let $E$ be a general principal bundle over $\mathcal{M}$ and $\omega$ a principal connection on $E$. Then for all $e_0\in E$, $ker(\omega_e)$ is transverse to the fiber (and of dimension $\dim(\mathcal{M})$) by lemma \ref{decompoTangent}. As a consequence, there exists a local section of $E$ around $x_0:= \pi_E(e_0)$ such that $e(x_0) = e_0$ and $\omega_{e_0}\circ\dd_{x_0}e_0 = 0$. Moreover, since $\dim(\dd_{x_0}e_0(T_{x_0}\mathcal{M}))=\dim(\mathcal{M})= \dim(ker(\omega_{e_0}))$, we have $\dd_{x_0}e_0(T_{x_0}\mathcal{M}) = ker(\omega_{e_0})$.
\end{remark}

We can use lemma \ref{computationLinear} to compute the connection in the general case.
\begin{coro}
\label{computationLinearGeneral}
We use the notations of proposition \ref{principalToLinear}. Let $e$ be a local smooth section of $E$ around $x_0\in \mathcal{M}$. Let $X \in T_{x_0}\mathcal{M}$ and $s$ be a section of $\mathcal{F}$ on some open neighborhood $U$ of $x_0$ such that $s(x) = [(e(x), v(x))]$ with $v: U\rightarrow V$ smooth. Then we have $\nabla_X s = [e(x_0), \dd_{x_0}v(X) + \dd_{1}\rho \omega_{e(x_0)}( \dd_{x_0}e(X)) v(x_0)]$.
\end{coro}
\begin{proof}
We define $e'$ smooth section of $E$ on a neighborhood $U$ of $x_0$ such that $e'(x_0) = e(x_0)$ and $\dd_{x_0}e' (T_{x_0}\mathcal{M}) = ker(\omega_{e(x_0)})$ ($e'$ exists by remark \ref{sectionPara}). We define $g: U\rightarrow G$ the unique smooth map such that $e'= e\cdot g$ (in particular $g(x_0) = 1$). By the chaine rule, we have:
\begin{align*}
\dd_{x_0} e'(x) &= \dd_{1}i_{e(x_0)}(\dd_{x_0}g(X)) + \dd_{x_0}e(X)\\
\end{align*}
But by definition of $e'(x_0)$, $\omega_{e(x_0)}(\dd_{x_0} e'(x)) = 0$. On the other hand:
\begin{align*}
\omega_{e(x_0)}(\dd_{x_0} e'(x)) &= \dd_{x_0}g (X) + \omega_{e(x_0)}(\dd_{x_0}e(X))\\
\end{align*}
therefore $\dd_{x_0}g (X) = -\omega_{e(x_0)}(\dd_{x_0}e(X))$. We can now compute:
\begin{align*}
s(x) &= [(e(x), v(x))] = [(e(x)\cdot g(x), \rho(g(x)^{-1})v(x))]\\
\nabla_X s (x_0) &= [(e(x_0)\cdot g(x_0), \dd_{x_0}(\rho (g^{-1})v)(X)] \quad \quad\text{(by proposition \ref{principalToLinear})}\\
&= [(e(x_0), \dd_{x_0}(v)(X) - \dd_1\rho \dd_{x_0}g(X)v(x_0))] \quad \quad \text{(chain rule)}\\
&= [(e(x_0), \dd_{x_0}(v)(X) + \dd_1\rho \omega_{e(x_0)}(\dd_{x_0}e(X)) v(x_0))]
\end{align*}
\end{proof}

The construction of proposition \ref{principalToLinear} behaves well with respect to the tensorial product of vector bundle as we see in the following proposition:
\begin{prop}
\label{computationTensorProduct}
Let $E$ be a $G$ principal bundle with a principal connection $\omega$, $\rho_1$ be a representation of $G$ over $V_1$ and $\rho_2$ a representation of $G$ over $V_2$. Let $\mathcal{F}_1$ be the vector bundle associated to $E$ with the representation $\rho_1$ and $\mathcal{F}_2$ the vector bundle associated to $E$ with the representation $\rho_2$. We denote by $\nabla_1$ (resp. $\nabla_2$) the linear connection on $\mathcal{F}_1$ (resp. $\mathcal{F}_2$) obtained from $\omega$ thanks to proposition \ref{principalToLinear}. The bundle $\mathcal{F}_1\otimes\mathcal{F}_2$ is naturally associated to $E$ with the representation $g\mapsto \rho_1(g)\otimes\rho_2(g)$. We denote by $\tilde{\nabla}$ the connection on $\mathcal{F}_1\otimes\mathcal{F}_2$ obtained from $\omega$ by proposition \ref{principalToLinear}. Then we have \[ \nabla_1\otimes \nabla_2 = \tilde{\nabla} \]
\end{prop}
\begin{proof}
Let $x_0\in \mathcal{M}$, let $s$ be a local section of $E$ on an open neighborhood $U$ of $x_0$ such that $\omega \circ \dd_{x_0} s = 0$ (it is always possible to find such a section by remark \ref{sectionPara}). Let $f_1$ (resp. $f_2$) be a smooth section of $\mathcal{F}_1$ (resp. $\mathcal{F}_2$) on $U$. We write $f_1 = [(s, g_1)]$ and $f_2 = [(s,g_2)]$ with $g_i: U \rightarrow V_i$ smooth. Then we have $f_1\otimes f_2(x) = [(s(x), g_1(x)\otimes g_2(x))]$. Let $X\in T_{x_0}\mathcal{M}$. We use lemma \ref{computationLinear} to compute
\begin{align*}
\tilde{\nabla}_X \left(f_1\otimes f_2\right) (x_0) &= [(s(x_0), \dd_{x_0}(g_1\otimes g_2)(X))] \\
&= [(s(x_0), \dd_{x_0}g_1(X) \otimes g_2(x_0) + g_1(x_0)\otimes \dd_{x_0}g_2(X))] \text{\hspace{15pt} chain rule and bilinearity of $\otimes$}\\
&= [(s(x_0), \dd_{x_0}g_1(X))]\otimes[(s(x_0), g_2(x_0))] + [(s(x_0), g_1(x_0))]\otimes[(s(x_0), \dd_{x_0}g_2(X))] \\
&= (\nabla_1)_X f_1(x_0) \otimes f_2(x_0) + f_1(x_0) \otimes (\nabla_2)_X f_2(x_0) \text{\hspace{15pt} by lemma \ref{computationLinear}}
\end{align*}
The equality is true for pure product sections and we use linearity to conclude.
\end{proof}

Finally, we give a kind of reverse construction of the previous one when $E$ is the principal bundle of frames of a vector bundle $\mathcal{F}$ of rank $n$ (real or complex, we consider the complex case here).
\begin{prop}
\label{linearToPrincipal}
We assume that we have a linear connection $\nabla$ on $\mathcal{F}$ (complex vector bundle of rank $n$ over a manifold $\mathcal{M}$ of positive dimension). We denote by $E$ the $GL(n,\C)$ principal bundle of frames.
There exists a unique one form $\omega$ on $E$ such that for all local frames $(e_1, ..., e_n)$ around $x$ and all $X\in T_x\mathcal{M}$, \[\omega_{e_1(x),...,e_n(x)}(\dd_x (e_1,...,e_n)(X)) = \underset{e_1(x),...,e_n(x)}{\text{Mat}}(\nabla_X e_1 (x), ... ,\nabla_X e_n(x)) \in \mathcal{M}_n(\C)\]
where $\underset{f_1,...,f_n}{\text{Mat}}(a_1,...,a_n)$ is the unique matrix $M$ such that for all $i\in \llbracket 1, n\rrbracket$, $a_i = \sum_{k=1}^n M_{k,i}f_k$. Moreover, $\omega$ is a principal connection on $E$.
\end{prop}
\begin{proof}
For $x \in \mathcal{M}$, we denote by $\Gamma_x(E)$ the set of local smooth section of $E$ defined on some neighborhood of $x$.
To prove the uniqueness, it is enough to remark that for a frame $(f_1,...,f_n)$ at $x$, the set $\left\{ \dd_x (e_1,...,e_n)(X): (e_1,...,e_n)\in \Gamma_x(E) \text{ with }(e_1,...,e_n)(x)=(f_1,...,f_n), X\in T_{x}\mathcal{M}\right\}$ generates $T_{(f_1,...,f_n)}E$. Indeed, by a simple construction in a local trivialization around $x$, we can show that $\left\{ \dd_x (e_1,...,e_n)(X): (e_1,...,e_n)\in \Gamma_x(E) \text{with} (e_1,...,e_n)(x)=(f_1,...,f_n), X\in T_{x}\mathcal{M}\right\}$ is exactly $T_{(f_1,...,f_n)}E\setminus ker (\dd_{(f_1,...,f_n)}\pi_E)$ (so we have the desired conclusion as soon as $\mathcal{M}$ has positive dimension). The existence follows from the observation that both sides of the equality are linear and if $\dd_x (e_1,...,e_n)(X) = \dd_x (e_1',...,e_n')(Y)$, then $(\nabla_X e_1 (x), ... ,\nabla_X e_n(x)) = (\nabla_Y e_1' (x), ... ,\nabla_Y e_n'(x))$.
Now we have to check that $\omega$ defines a principal connection on the bundle of frames. In particular, for $(f_1,...,f_n) \in E$ and $g\in GL(n,\R)$, $Ad_g (R_g^*\omega)_{f_1,...,f_n} = \omega_{f_1,...,f_n}$. By the remark in the proof of uniqueness, it is enough to check it on the vectors on the form $\dd_x (e_1,...,e_n)(X)$ where $(e_1,...,e_n)(x) = (f_1,...,f_n)$. We define $(e_1',...,e_n'):= (e_1, ..., e_n)\cdot g$ for $g \in GL(n,\C)$ and by definition of the right action of $GL(n,\C)$ on $E$, we have $e_i' = \sum_{k=1}^n g_{k,i} e_k$. Then we have $\dd_{(f_1,...,f_n)}R_g \dd_x(e_1,...,e_n)(X) = \dd_x(e_1',..., e_n')(X)$. We use that to compute:
\begin{align*}
R_g^*\omega(\dd_x(e_1,...,e_n)(X)) &= \omega_{(f_1,...,f_n)\cdot g} (\dd_x (e_1',...,e_n')(X)) \\
&= \underset{(f_1, ..., f_n)\cdot g}{\text{Mat}} \left(\nabla_X e_1',...,\nabla_X e_n'\right)\\
 &= \underset{(f_1, ..., f_n)\cdot g}{\text{Mat}} \left(\sum_{k=1}^n(g_{k,1}\nabla_X e_k,...,g_{k,n}\nabla_X e_k)\right)\\
&= g^{-1}\underset{(f_1,...,f_n)}{\text{Mat}}(\nabla_X e_1,..., \nabla_X e_n) g\\
&= (Ad_g)^{-1} \omega(\dd_x(e_1,...,e_n)(X))
\end{align*}
The second property to check is that for $(f_1,...,f_n) \in E$ and $g: (-1,1)\rightarrow GL(n,\C)$ smooth with $g(0)=Id$, we have $\omega(\frac{\dd}{\dd t}_{|t=0}(f_1,...,f_n)\cdot g(t)) = \frac{\dd}{\dd t}_{|t=0}g(t)$.
To see that, we take $(x_1,...,x_n)$ smooth coordinates around $x$ with $(x_1,...,x_n)(x)=0$ and $(e_1,...,e_n)$ a local frame around $x$ such that $(e_1,...,e_n)(x) = (f_1,...,f_n)$. We define the local frame $(e_1',...,e_n'):= (e_1,...,e_n)\cdot g(x_1)$. We compute (using the chain rule)
\begin{align*}
\dd_x(e_1',...,e_n')(\partial_{x_1}) = \frac{\dd}{\dd t}_{|t=0} \left((f_1,...,f_n)\cdot g(t)\right) + \dd_x (e_1,...,e_n)(\partial_{x_1}) \\
\end{align*}
Moreover, we have
\begin{align*}
e_i' = \sum_{k=1}^n g_{k,i}(x_1) e_k
\end{align*}
Therefore $ \nabla_{\partial_{x_1}} e_i' (x) = \sum_{k=1}^n \left(\frac{\dd}{\dd t}_{|t=0} g_{k,i}\right)f_k + \nabla_{\partial_{x_1}}e_k(x)$ and using the definition, we see that\newline $\omega_{(f_1,...,f_n)} (\dd_x(e_1', ..., e_n')(\partial_{x_1})) = \omega_{(f_1,...,f_n)}(\dd_x (e_1,...,e_n)(\partial_{x_1})) + \frac{\dd}{\dd t}_{|t=0}g(t)$. Eventually,we deduce:
\[ \omega_{(f_1,...,f_n)}\left(\frac{\dd}{\dd t}_{|t=0} \left((f_1,...,f_n)\cdot g(t)\right)\right) = \frac{\dd}{\dd t}_{|t=0}g(t). \]
\end{proof}

\begin{prop}The constructions of proposition \ref{principalToLinear} and \ref{linearToPrincipal} are inverse of one another.
\end{prop}
\begin{proof}
Let $Y$ be a local smooth section of $\mathcal{F}$ around $x_0\in \mathcal{M}$ and $X \in T_{x_0} \mathcal{M}$. We denote by $\omega$ the principal connection on $E$ obtained from $\nabla$ by going through the construction of proposition \ref{linearToPrincipal} and $\overline{\nabla}$ the linear connection obtained from $\omega$ by proposition \ref{principalToLinear}. We prove that $\nabla_X Y = \overline{\nabla}_X Y$. To do so we choose a local basis $(e_1,...,e_n)$ around $x_0$ such that $(\nabla_X e_1,...\nabla_X e_n) = 0$ (it is always possible to construct such a basis by working in a local trivialization around $x$). We denote by $f$ the function defined from a neighborhood of $\pi_E^{-1}(\left\{x_0\right\})$ to $\C^n$ by $f(a) = \underset{a}{\text{Mat}}Y(\pi_E(a))$. We define $a_0:= (e_1,...,e_n)(x_0)$. 

By definition of $\omega$, we have $\omega_{a_0}\dd_{x_0}(e_1,...,e_n)(X) = 0$. If we write $Y = [ ((e_1,...,e_n), (Y_1,...,Y_n)) ]$ where $Y_i$ are the coordinates of $Y$ in the local basis $(e_1,...,e_n)$, we can use lemma \ref{computationLinear} to deduce that $\overline{\nabla}_XY(x_0) = [(a_0, (\dd_{x_0}Y_1(X), ... , \dd_{x_0}Y_n(X)))] = \sum_{k=1}^n X(Y_k)(x_0)e_k(x_0)$.

On the other hand 
\begin{align*}
\nabla_X Y &= \sum_{k=1}^n X(Y_k)(x_0)e_k + Y_k \nabla_X e_k(x_0)\\
&= \sum_{k=1}^n X(Y_k)(x_0)e_k \text{ by definition of $(e_1,...,e_k)$}\\
&= \overline{\nabla}_X Y
\end{align*}
We conclude that $\nabla = \overline{\nabla}$.

We also have to prove that if $\omega$ is a principal connection on $E$ and if $\nabla$, linear connection on $\mathcal{F}$ is obtained by proposition \ref{principalToLinear}, then the principal connection $\tilde{\omega}$ on $E$ obtained from $\nabla$ by proposition \ref{linearToPrincipal} is equal to $\omega$. Since we already know that for $x_0\in \mathcal{M}$, $a \in E$ and $h\in \mathfrak{g}$, $\omega \dd_{1} i_{a} (h) = \tilde{\omega} \dd_{1} i_{a} (h) = h$ (by definition of a principal connection) and since $T_{a}E = ker{\omega_a} \oplus \dd_{1}i_a(\mathfrak{g})$ (see lemma \ref{decompoTangent} and remark \ref{egaliteEns}), it is enough to prove that $\tilde{\omega} = 0$ on $ker(\omega_a)$. By remark \ref{sectionPara}, there exists a smooth local section $(e_1,...,e_n)$ on an open neighborhood $U$ around $x_0$ such that $(e_1,...,e_n)(x_0)=a$ and $\dd_{x_0}(e_1,...,e_n)(T_{x_0}\mathcal{M}) = ker(\omega_a)$. We are reduced to proving that $\tilde{\omega_a}\circ\dd_{x_0}(e_1,...e_n) = 0$. By definition of $\tilde{\omega}$, this is the same as proving that for all $i\in\llbracket 1, n\rrbracket$, $\nabla e_i = 0$. By definition of $\nabla$ and lemma \ref{computationLinear} (use the fact that $\dd_{x_0}(e_1,...,e_n)(T_{x_0}\mathcal{M}) = ker(\omega_a)$), it is the same as proving that if $f_i: U\rightarrow \R^n$ is such that $e_i = [(e_1,...,e_n), f_i]$ then $\dd_{x_0} f_i = 0$. This last fact is obvious since $f_i$ are the coordinates of $e_i$ in the local basis $(e_1,...,e_n)$.
\end{proof}

\begin{prop}[pull back of a principal connection]
\label{pullBack}
Let $A$ be a $G_A$ principal bundle over $\mathcal{M}$ and $B$ a $G_B$ principal bundle over the same manifold $\mathcal{M}$. We assume that $G_A$ is an embedded Lie subgroup of $G_B$ and that we have an embedding $f:A\rightarrow B$ (with $\pi_B\circ f=\pi_A$) such that for all $g\in G_A$ and $a\in A$, $f(a\cdot g) = f(a)\cdot g$. For every principal connection $\omega$ on $B$ such that for all $a\in A$, $ker(\omega_{f(a)})\subset T_{f(a)} f(A)$, the one form $f^*\omega$ is a principal connection on $A$.
\end{prop}

\begin{proof}
We first have to check that $f^*\omega$ has values in $\mathfrak{g}_A$. To prove that, we remark that for all $a\in A$, $Ran(\dd_a f) = ker\omega_{f(a)} \oplus \dd_1 i_{f(a)}\mathfrak{g}_A$. Indeed, the right-hand side is included into the left-hand side and both sides have the same dimension ($\dim(Ran(\dd_a f)) = \text{dim}A = \dim(\mathcal{M})+\dim(G_A)$). Moreover, $\omega( ker\omega_{f(a)} \oplus \dd_1 i_{f(a)}\mathfrak{g}_B) = \mathfrak{g}_A$ so $f^*\omega$ has value in $\mathfrak{g}_A$. The two properties of principal connection for $f^*\omega$ follow directly for the corresponding one for $\omega$.
\end{proof}

\begin{remark}
\label{conditionPullBack}
The condition $ker(\omega_{f(a)})\subset T_{f(a)} f(A)$ takes a particularly simple form when $B$ is the frame bundle of a vector bundle $\mathcal{F}$ and $\omega$ comes from a linear connection $\nabla$ on $\mathcal{F}$ (by proposition \ref{linearToPrincipal}). Indeed, we have the following equivalence:
 $ker(\omega_{f(a)})\subset T_{f(a)} f(A)$ if and only if for all $a\in A$, there exists $(e_1,...,e_n)$ a local section of $f(A)$ around $x_0:=\pi_B(f(A))$ such that $(e_1,...,e_n)(x_0)=a$ and for all $X\in T_{x_0}A$, $(\nabla_X e_1,..., \nabla_X e_n)=0$.
The idea of the proof is the following:
Let $a\in A$ and $x_0:=\pi_B(f(a))$. Assume $ker(\omega_{f(a)})\subset T_{f(a)} f(A)$, then we can find a submanifold $C$ of $f(A)$ of dimension $\text{dim}\mathcal{M}$ containing $f(a)$ with tangent space $ker(\omega_{f(a)})$ at $f(a)$. We know that $ker(\omega_{f(a)})\cap ker(\dd_{f(a)}\pi_B)=\left\{0\right\}$, as a consequence $C$ defines a section of $f(A)$ in a small neighborhood of $f(a)$ this section is the one we are looking for. Reciprocally, if we have a section $(e_1,...,e_n)$ with the required properties, $\dd_x(e_1,...,e_n)(T_{x_0}\mathcal{M})\subset ker(\omega_{f(a)})$ and has the same dimension, so $ker(\omega_{f(a)})=\dd_x(e_1,...,e_n)(T_{x_0}\mathcal{M}) \subset T_{f(a)} f(A)$.
\end{remark}

A typical situation where the pull back appears is the following:
\begin{prop}
\label{pullBackFrame}
Let $A$ be a $G$ principal bundle and $\rho:G\rightarrow GL(\C^n)$ a representation of $G$ on $\C^n$ (we can also replace $\C^n$ by $\R^n$) which is an embedding of Lie-groups. Let $\mathcal{F}$ be the vector bundle associated to $A$ with the representation $\rho$. Let $\omega$ be a principal connection on $A$ and $\nabla$ the connection on $\mathcal{F}$ constructed by proposition \ref{principalToLinear}. Let $E$ be the principal bundle of frames on $\mathcal{F}$ and $\omega_E$ be the principal connection on $E$ given by proposition \ref{linearToPrincipal}. Then if we define the embedding: 
\[
f:
\begin{cases}
A \rightarrow E\\
a \mapsto ([a,e_1], [a,e_2], ..., [a,e_n])
\end{cases}
\]
where $(e_1, ..., e_n)$ is the canonical basis of $\C^n$. Then $f^*\omega_E = \dd_1\rho \circ \omega$.
\end{prop}
\begin{proof}
First, we can check that for all $a\in A$ and $g\in G$, $f(a\cdot g) = ([a,e_1], ..., [a,e_n])\cdot\rho(g)$. In particular for $a\in A$, $\dd_1 (f\circ i_a) = \dd_a f \dd_1 i_a = \dd_1 i_{f(a)} \dd_1 \rho$. By the definition of a principal connection, we know that $\dd_1 \rho \omega_a \dd_1 i_a (h) = \dd_1 \rho(h)$ and also $ (\omega_E)_{f(a)} \dd_1 i_{f(a)}(\dd_1\rho(h)) = \dd_1\rho(h)$. We deduce that $\dd_1 \rho \omega_a \dd_1 i_a (h)= (\omega_E)_{f(a)} \dd_1 i_{f(a)}(\dd_1\rho(h))= (\omega_E)_{f(a)}\dd_a f \dd_1 i_a(h)$. Since $T_{\pi_A(a)}A = \dd_1 i_a (\mathfrak{g}) \oplus ker(\omega_a)$, we are now reduced to checking that $(\omega_E)_{f(a)} \dd_a f (ker(\omega_a)) = 0$. We take $s$ a local section of $A$ in an open neighborhood $U$ of $x_0$ such that $s(x_0)=a$ and $\dd_{x_0} s(T_{x_0}\mathcal{M}) = ker(\omega_a)$. Then by lemma \ref{computationLinear}, $\nabla [s, e_i] = 0$. We deduce that for $X\in T_{x_0}\mathcal{M}$, $(\nabla_X [s, e_1], ..., \nabla_X [s,e_n]) = 0$. Finally, using the definition of $\omega_E$ we deduce that $(\omega_E)_{f(a)} ( \dd_{x_0} f\circ s(X)) = \underset{f(a)}{\text{Mat}}(\nabla_X [s, e_1], ..., \nabla_X [s,e_n]) =0$. Therefore $(\omega_E)_{f(a)}\circ \dd_{x_0} f (ker(\omega_A)) = 0$
\end{proof}

In some cases of interest, the condition in remark \ref{conditionPullBack} is not satisfied. In these cases, the pull back is not a connection because the image is not contained in the Lie subalgebra $\mathfrak{g}_A$. The goal of the following proposition is to correct this by composing (to the left) by a projection on $\mathfrak{g}_A$. The main downside with this construction is that we have several choices for the projection leading to different connections. However, in cases we are interested in here, there is a particularly natural choice (see the remark after the proof).

\begin{prop}[Pull Back of a connection in more complicated cases]
\label{pullBack2}
Let $A$ be a $G_A$ principal bundle over $\mathcal{M}$ and $B$ a $G_B$ principal bundle over the same manifold $\mathcal{M}$. We assume that $G_A$ is an embedded Lie subgroup of $G_B$ and that we have an embedding $f:A\rightarrow B$ such that for all $g\in G_A$ and $a\in A$, $f(a\cdot g) = f(a)\cdot g$. For simplicity, we identify implicitly the Lie algebra of $A$ and the Lie algebra of $f(A)$. Assume that we have a subspace $V$ of $\mathfrak{g}_B$ such that:
\begin{itemize}
\item $V\oplus \mathfrak{g}_A = \mathfrak{g}_B$
\item $\forall g\in G_A, Ad_g(V) = V$
\end{itemize}
We denote by $q: \mathfrak{g}_B \rightarrow \mathfrak{g}_A$ the projection on $\mathfrak{g}_A$ with kernel $V$.
 For every principal connection $\omega$ on $B$, we can define the $\mathfrak{g}_A$-valued one form $\tilde{\omega}:= q\circ f^*\omega$. Then $\tilde{\omega}$ is a principal connection on $A$.
\end{prop}
\begin{proof}
For all $g\in G_A$, we have $Ad_g(\mathfrak{g}_A)\subset \mathfrak{g}_A$ and $Ad_g(V)\subset V$ so $Ad_g$ commutes with $q$.
Let $g \in G_A$ and $a\in A$, 
\begin{align*}
Ad_g R_g^*\tilde{\omega}_a&= Ad_g \tilde{\omega}_{a\cdot{g}}\circ \dd_aR_g \\
&= Ad_g\circ q (f^*\omega)_{a\cdot{g}}\circ \dd_aR_g\\
&= q\circ Ad_g \omega_{f(a\cdot g)} \circ \dd_{a\cdot g} f \circ \dd_aR_g \\
&= q\circ Ad_g \omega_{f(a)\cdot g} \circ \dd_{f(a)}R_g \circ \dd_a f(a) \\
&= q\circ Ad_g (R_g^* \omega)_{f(a)} \circ \dd_a f(a)\\
&= q \circ (f^* (Ad_g R_g^*\omega))_a\\
&= q \circ (f^*\omega)_a\\
&= \tilde{\omega}_a.
\end{align*}
The second property of principal connections is immediate since $q$ is the identity on $\mathfrak{g}_A$.
\end{proof}

\begin{remark}
\label{remarkKillingForm}
In some cases, there is a natural choice for $V$. We recall that the Killing form of $\mathfrak{g}_B$ is by definition the symmetric bilinear form $K(x,y) = \text{tr}(ad_x ad_y)$ (where $ad_x$, $ad_y$ are considered as endomorphisms of $\mathfrak{g}_B$). Note that the Killing form is invariant under every automorphism of Lie algebra of $\mathfrak{g}_B$. When $K$ is non degenerate, we say that $\mathfrak{g}_B$ is semisimple. If moreover, $K_{|_{\mathfrak{g}_A\times\mathfrak{g}_A}}$ is also non degenerate, we can consider the very natural choice $V:= \mathfrak{g}_{A}^\perp$ where $\mathfrak{g}_A^\perp$ is the orthogonal with respect to $K$ (note that the two non degeneracy conditions imply that $V\oplus\mathfrak{g}_A = \mathfrak{g}_B$). The fact that $Ad_g(V) = V$ for all $g\in G_A$ is then given by the invariance of the Killing form by the Lie algebra automorphism (preserving $\mathfrak{g}_A$ since $g\in G_A$) $Ad_g$.
\end{remark}

We conclude this section by the following proposition that we do not prove (the proof follows quite easily from the definitions)
\begin{prop}
\label{pullBackCovering}
Let $A$ be a $G_A$ principal bundle over $\mathcal{M}$ and $B$ be a $G_B$ principal bundle over the same manifold $\mathcal{M}$. We assume that we have a covering Lie group morphism $\tilde{f}: G_A\rightarrow G_B$ and a smooth covering map $f:A\rightarrow B$ such that $\pi_B\circ f = \pi_A$ and for all $g\in G_A$ and $a\in A$, $f(a\cdot g) = f(a)\cdot \tilde{f}(g)$. Then, if $\omega$ is a principal connection on $B$, $f^*\omega$ is a principal connection on $A$.
\end{prop}

\subsection{Spin connection, GHP connection and GHP operators}
For deeper geometric insight on the definitions of connections and operators, see \cite{harnett1990ghp}.
\subsubsection{Spin connection}
We denote by $\nabla$ the Levi-Civita connection on $\mathcal{M}$. Using proposition \ref{linearToPrincipal}, we can define a principal connection $\omega$ on the space of complex tangent frames $E$. 
We now want to apply proposition \ref{pullBack} (and remark \ref{conditionPullBack}) to the various principal bundles previously defined.
We have an embedding of principal bundles $f: \mathfrak{O} \rightarrow E$. We use the following proposition and proposition \ref{pullBack} (and remark \ref{conditionPullBack}) to show that $f^*\omega$ is a principal connection on $\mathfrak{O}$. We still call it $\omega$.
\begin{prop}
\label{normalBON}
Let $y_0 \in \mathcal{M}$.
For all $(f_0,...,f_3) \in \mathfrak{O}_{y_0}$, there exists a local smooth section $(e_0,...,e_3)$ (around $y_0$) of $\mathfrak{O}$ such that $(e_0,...,e_3)(y_0)=(f_0,...,f_3)$ and for all $X \in T_{y_0}\mathcal{M}, (\nabla_X e_0, ..., \nabla_X e_3) = 0$. 
\end{prop}
\begin{proof}
We define $(x_0, ..., x_3)$ local normal coordinates on a neighborhood of $y_0$ such that $(\partial_{x_0},...,\partial_{x_3})(y_0) = (f_0,...,f_3)$. Then we denote by $g_{i,j}$ the metric coefficients in these coordinates. Because coordinates are normal we have for all $X\in T_{y_0}\mathcal{M}$, for all $i\in\llbracket 0, 3\rrbracket$, $\nabla_{X}\partial_{x_i} = 0$, $g_{i,j}(y_0) = \delta_{i,j}$ and $g$ has vanishing first derivatives at $y_0$. We define $(e_0,...,e_3)$ as the Schmidt orthonormalization of $(\partial_{x_0},...,\partial_{x_3})$. As a consequence there is a smooth family of upper triangular invertible matrices $S(y)$ that we can express explicitly with respect to the coefficients $g_{i,j}$ such that $S(y_0) = Id$ and $(e_1,...,e_n)(y) = (\partial_{x_0},...,\partial_{x_3})(y)S(y)$. Then we have, for $X\in T_{y_0}\mathcal{M}$, $(\nabla_X e_1, ..., \nabla_X e_n) = (\nabla_X \partial_{x_0},..., \nabla_X \partial_{x_3}) + (\partial_{x_0},...,\partial_{x_3})(y) X(S)(y_0)$ where $X(S)(y_0)$ is the derivation coefficient by coefficient. Because all the derivatives of the metric coefficients $g_{i,j}$ vanish at $y_0$, we have $X(S)(y_0) = 0$ and $(\nabla_X e_1, ..., \nabla_X e_n)= 0$.
\end{proof}

We can then use proposition \ref{pullBackCovering} to define a connection $\omega$ on $\mathfrak{S}$ by pulling back the principal connection on $\mathfrak{O}$. We can also define a linear connection $\nabla$ on the spinor bundle $\mathcal{S}$ and on $\overline{\mathcal{S}}$ by using proposition \ref{principalToLinear}. Concretely, note that for a real vector field $X$ and a section $\overline{a}$ of $\overline{\mathcal{S}}$, we have the equality $\nabla_X \overline{a} = \overline{\nabla_X a}$ (it follows from the definitions). Therefore, by $\C$-linearity, if $X$ is a complex vector field, $\nabla_X \overline{a} = \overline{\nabla_{\overline{X}} a}$. We used the same notation $\nabla$ and $\omega$ for connections on different bundles (and we use the context to remove ambiguity). To show that all these definitions are coherent, it is useful to prove the following proposition:

\begin{prop}
\label{idLCspin}
We have $j^*\nabla = \nabla \otimes \nabla$ where the left hand side is the Levi-Civita connection on $T_{\C}\mathcal{M}$ and the right hand side is the connection on $\mathcal{S}\otimes \overline{\mathcal{S}}$.
\end{prop}

\begin{proof}
First note that $\nabla \otimes \nabla$ is the same connection as the one obtained by proposition \ref{principalToLinear} applied to the bundle $\mathcal{S}\otimes \overline{\mathcal{S}}$ seen as associated to $\mathfrak{S}$ (see proposition \ref{computationTensorProduct}).
Let $x_0 \in \mathcal{M}$. Using proposition \ref{normalBON}, we define a local smooth section $(e_0,e_1,e_2,e_3)$ of $\mathfrak{O}$ around $x_0$ such that $(\nabla e_0 (x_0), ..., \nabla e_3 (x_0)) = 0$. In particular, by definition of the connection on $\mathfrak{O}$, $\omega \circ \dd_{x_0} (e_0, ..., e_3) = 0$. Let $s$ be a local smooth section of $\mathfrak{S}$ around $x_0$ such that $p \circ s = (e_0,..., e_3)$. By definition of the connection on $\mathfrak{S}$, we have $\omega \circ \dd_{x_0} s = 0$. Moreover, the trivialization induced by $s$ on $\mathfrak{S}$ and the one induced by $(e_0,...,e_3)$ on $\mathfrak{O}$ are compatible. Then $j$ written in these local trivializations is just the map $id_U\times i$ where $i_0:\C^2\otimes \overline{\C}^2 \rightarrow \C^4$ is an isomorphism (defined earlier).
We take $Z_1$ a local smooth section of $\mathcal{S}\otimes \overline{\mathcal{S}}$ on a small neighborhood $U$ of $x_0$, then $Z_2 = j(Z_1)$ is a local smooth section of $T_{\C}\mathcal{M}$ around $x_0$. We have $Z_1 = [(s,\tilde{Z}_1)]$ with $\tilde{Z}_1: U \rightarrow \C^2\otimes \overline{\C}^2$ and $Z_2 = [((e_0,...,e_3),\tilde{Z_2})]$ with $\tilde{Z}_2 := i\circ \tilde{Z}_1$. Let $X\in T_{x_0}\mathcal{M}$. By the first remark in this proof and the definition of the connection on an associated bundle (proposition \ref{principalToLinear}), we have that $(\nabla\otimes \nabla)_X Z_1 = [(s,\dd_{x_0} \tilde{Z}_1(X))]$ (it uses the fact that $\omega \circ \dd_{x_0} s = 0$ and lemma \ref{computationLinear}). Moreover, similarly we have $\nabla_X Z_2 = [((e_0,...,e_3),\dd_{x_0} \tilde{Z}_2(X))] = [((e_0,...,e_3),i_0\dd_{x_0}\tilde{Z_1}(X))]$ (it uses the fact that $\omega \circ \dd_{x_0} (e_0, ..., e_3) = 0$, lemma \ref{computationLinear} and the linearity of $i_0$). Then we conclude that $\nabla_X j\circ Z_1 = j (\nabla\otimes \nabla)_X Z_1$. Since it is true for all local smooth sections $Z_1$, all $X\in T_{x_0}\mathcal{M}$ and all $x_0\in \mathcal{M}$, we have proved the proposition.
\end{proof}

\begin{remark}
In this remark, we use the identification $j$ implicitly.
Note that $\nabla$ is not the only connection with the property $\nabla\otimes\nabla = \nabla_{LC}$ (at least locally). Indeed, let $U$ be an open set on which we have a smooth spin frame $(o,\iota)$ with $\epsilon(o,\iota) = 1$. Then we can define a connection $\nabla'$ on $\mathcal{S}_{|_U}$ such that for all real vector field $X$ on $U$:
\begin{align*}
\nabla'_X o &= (a(X)+i\mu(X)) o + b(X)\iota\\
\nabla'_X \iota &= c(X) o +(- a(X) + i\mu(X)) \iota 
\end{align*}
where $a$, $b$ and $c$ are complex valued linear forms on $U$ and $\mu$ is a real valued linear form on $U$.
Moreover, we define $a$,$b$ and $c$ as follows:
\begin{equation}
\label{conditionsLC}
\begin{aligned}
b(X) &= -g((\nabla_{LC})_X(o\otimes \overline{o}), o\otimes\overline{\iota}) \\
c(X) &= -g((\nabla_{LC})_X(\iota\otimes\overline{\iota}), \iota\otimes \overline{o})\\
a(X) &= \frac{1}{2}\left(g((\nabla_{LC})_X(o\otimes \overline{o}), \iota\otimes\overline{\iota})-g((\nabla_{LC})_X (o\otimes\overline{\iota}), \iota\otimes\overline{o})\right)
\end{aligned}
\end{equation}
Conditions \eqref{conditionsLC} are necessary to have $\nabla\otimes\nabla = \nabla_{LC}$ and we now prove that they are sufficient (in the definition of $\nabla'$ we can chose freely any real linear form $\mu$ hence the lack of uniqueness). 
Note that $o \otimes \overline{o}$, $\iota\otimes\overline{\iota}$, $o\otimes \overline{\iota}$, $\iota\otimes \overline{o}$ is a normalized null tetrad on $U$. Therefore we have (using properties of the Levi-Civita connection):
\begin{align*}
(\nabla_{LC})_X o\otimes\overline{o} &= 2\Re(a(X)) o\otimes\overline{o}+b(X) \iota\otimes\overline{o} +\overline{b(X)} o\otimes\overline{\iota}\\
(\nabla_{LC})_X \iota\otimes\overline{\iota} &= -2\Re(a(X))\iota\otimes\overline{\iota} + c(X)o\otimes\overline{\iota} + \overline{c(X)}\iota\overline{o}\\
(\nabla_{LC})_X o\otimes \overline{\iota} &= \overline{c(X)}o\otimes\overline{o} + b(X)\iota\otimes\overline{\iota} + 2i\Im(a(X))o\otimes\overline{\iota}\\
(\nabla_{LC})_X \iota\otimes\overline{o} &= c(X)o\otimes\overline{o} + \overline{b(X)}\iota\otimes\overline{\iota} - 2i\Im(a(X))\iota\otimes\overline{o}
\end{align*}
To check that $\nabla'\otimes\nabla' = \nabla_{LC}$, it is enough to prove that check the equality on this tetrad. It follows from the definition of $\nabla'$.
\end{remark}

The previous remark shows that $\nabla$ is not completely determined by proposition \ref{idLCspin}. However, we also have the following proposition:
\begin{prop}
\label{nablaEpsilonZero}
We have $\nabla \epsilon = 0$
\end{prop}
\begin{proof}
By definition, $\nabla \epsilon = 0$ if for all $x_0\in \mathcal{M}$, $X\in T_{x_0}\mathcal{M}$ and $a$,$b$ spinor fields defined on a neighborhood $U$ of $x_0$, we have:
\begin{align*}
X(\epsilon(a,b))(x_0) = \epsilon(\nabla_X a, b)(x_0) + \epsilon(a,\nabla_X b)(x_0) .\\
\end{align*}
We denote by $\omega_{\mathfrak{S}}$ the principal connection on $\mathfrak{S}$ previously defined (and used to define the spin connection $\nabla$)
Let $s$ be a smooth local section of $\mathfrak{S}$ such that $(\omega_{\mathfrak{S}})_{s(x_0)}\dd_{x_0} s = 0$ (which exists by remark \ref{sectionPara}). Let $\tilde{a}:U\rightarrow \C^2$ and $\tilde{b}:U\rightarrow \C^2$ smooth be such that $a(x)=[(s(x),\tilde{a}(x))]$ and $b(x) = [(s(x), \tilde{b}(x))]$. By lemma \ref{computationLinear}, we have 
\begin{align*}
\nabla_X a(x_0) &= [(s(x_0), \dd_{x_0}\tilde{a}(X))]\\
\nabla_X b(x_0) &= [(s(x_0), \dd_{x_0}\tilde{b}(X))]
\end{align*}
Moreover, by definition of $\epsilon$ (see proposition \ref{defEpsilon}), we have $\epsilon(a,b)(x) = \det(\tilde{a}(x), \tilde{b}(x))$. Therefore by bilinearity of the determinant: 
\begin{align*}
X(\epsilon(a,b))(x_0) &= \det(\dd_{x_0}\tilde{a}(X),\tilde{b}(x_0)) + \det(\tilde{a}(x_0),\dd_{x_0}\tilde{b}(X))\\
 &= \epsilon(\nabla_X a, b)(x_0) + \epsilon(a,\nabla_X b)(x_0).
\end{align*}
\end{proof}

\begin{remark}
\label{computationNabla}
It is very useful to note that the connection $\nabla$ on $\mathcal{S}$ is completely determined by proposition \ref{idLCspin} and proposition \ref{nablaEpsilonZero} (in particular, we could have used these properties as a definition of the connection $\nabla$ on $\mathcal{S}$). Indeed, if we fix a local basis $(o,\iota)$ of $\mathcal{S}$ such that $\epsilon(o, \iota) = 1$ and $X$ a vector field on $\mathcal{M}$, we can write:
\begin{align*}
\nabla_X o &= a(X)o + b(X)\iota \\
\nabla_X \iota &= c(X)o + d(X)\iota
\end{align*}
Using proposition \ref{nablaEpsilonZero} and $\epsilon(o,\iota) = 1$, we have $a(X) = -d(X)$. By proposition \ref{idLCspin}, we also have:
\begin{align*}
b(X) &= -g(\nabla_X (o\otimes \overline{o}), o\otimes \overline{\iota})\\
c(X) &= -g(\nabla_X (\iota\otimes \overline{\iota}), \iota \otimes \overline{o})\\
a(X) &= \frac{1}{2}\left(g(\nabla_X (o\otimes \overline{o}), \iota\otimes \overline{\iota})-g(\nabla_X (o\otimes \overline{\iota}), \iota \otimes \overline{o})\right) 
\end{align*}
The coefficients of the one forms $a$, $-b$ and $c$ on the null basis $(l,n,m,\overline{m}) := (o\otimes\overline{o}, \iota\otimes\overline{\iota}, o\otimes \overline{\iota}, \iota\otimes\overline{o})$ are called spin coefficients. For example, following the notation in \cite{chandrasekhar1998mathematical} chapter 1 (286) for the spin coefficients:
\[\begin{array}{cccc}
\kappa = -b(l) & \tau = -b(n) & \sigma = -b(m)& \rho = -b(\overline{m}) \\
\pi = c(l) & \nu = c(n) & \mu = c(m) & \lambda = c(\overline{m})\\
\epsilon = a(l) & \gamma = a(n) & \beta = a(m) & \alpha = a(\overline{m})
\end{array}\]
\end{remark}

\begin{remark}
\label{vanishingSpinCoeff}
If we require that $(o,\iota)$ is a local section of $\mathcal{A}_0$ (and as usual $(l,n,m) := d(o, \iota)$), we have that $l$ and $n$ are pregeodesic and therefore, by remark \ref{computationNabla}, we have:
\begin{align*}
b(l) &= -g(\nabla_{l}l, m)=0\\
c(n) &= -g(\nabla_{n}n, \overline{m}) = 0
\end{align*}
Moreover, since $l$ and $n$ are shear-free and $\Re(m), -\Im(m)$ is an orthonormal family of $l^\perp$ and $n^\perp$, we have
\begin{align*}
0 =& \frac{1}{2}\left(g(\nabla_{\Im(m)}l, \Im(m))- g(\nabla_{\Re(m)}l, \Re(m))\right) - \frac{i}{2}\left(g(\nabla_{\Re(m)} l, \Im(m))+ g(\nabla_{\Im(m)}l, \Re(m))\right)\\
=& -\frac{1}{2} g(\nabla_m l, m)\\
\end{align*}
and similarly
\begin{align*}
0=& -\frac{1}{2} g(\nabla_{m} n, m) \\
=& -\frac{1}{2}g(\nabla_{\overline{m}} n, \overline{m}) &\text{ (complex conjugate)}
\end{align*}
Therefore we have:
\begin{align*}
b(m) =& 0\\
c(\overline{m}) =& 0
\end{align*}
In the spin-coefficient formalism, this corresponds to the vanishing of $\kappa$, $\lambda$, $\sigma$ and $\nu$.
\end{remark}
\begin{remark}
\label{computationOmega}
We can also compute concretely the connection $\omega$ on $\mathfrak{S}$ using proposition \ref{pullBackFrame}: Indeed, we have a natural embedding of $\mathfrak{S}$ in the set of frames $(o,\iota)$ of $\mathcal{S}$ given by the map $f$ of proposition \ref{pullBackFrame} (note that the representation of $SL(2,\C)$ into $GL(\C^2)$ is just the natural embedding so $\dd_1 \rho$ is just the inclusion of $sl(2,\C)$ into $M_2(\C)$). The map $f$ identifies $\mathfrak{S}$ with the set of frames $(o,\iota)$ such that $\epsilon(o, \iota)=1$. Therefore, proposition \ref{pullBackFrame} tells us that for a local smooth section $s$ of $\mathfrak{S}$ identified with the local frame $(o,\iota)$ and $X\in T_{x_0}\mathcal{M}$, $\omega( \dd_{x_0} s(X)) = \underset{(o,\iota)(x_0)}{\text{Mat}}(\nabla_X o, \nabla_X \iota) = \begin{pmatrix}a(X) & c(X) \\ b(X) & -a(X)\end{pmatrix}\in sl(2,\C)$.
\end{remark}

\subsubsection{GHP connection}
We now define the connection on $\mathcal{A}_0$ and $\mathcal{N}_0$ using proposition \ref{pullBack2}. We cannot use proposition \ref{pullBack} since the condition of remark \ref{conditionPullBack} is clearly not satisfied.
The covering map $\tilde{p}: SL(2,\C) \rightarrow SO^+(1,3)$ gives a Lie algebra isomorphism $\dd_{Id} \tilde{p}: sl(2,\C) \rightarrow so(1,3)$. 

We also have an embedding of principal bundles
\[f:\begin{cases}
\mathcal{N}_0 \rightarrow \mathfrak{O}\\
(l,n,m) \mapsto \left(\frac{l+n}{\sqrt{2}}, \sqrt{2}\Re(m), -\sqrt{2}\Im(m), \frac{l-n}{\sqrt{2}}\right)
\end{cases}
\]
associated\footnote{Meaning that for all $x\in \mathcal{N}_0$ and $g\in \C^*$ we have $f(x\cdot g) = f(x)\cdot\tilde{f}(g)$} to the embedding of Lie groups
\[
\tilde{f}: 
\begin{cases}
\C^* \rightarrow SO^+(1,3)\\
|z|e^{i\theta} \mapsto \begin{pmatrix} \frac{1}{2}(|z|+|z|^{-1}) & 0 & 0 & \frac{1}{2}(|z|-|z|^{-1})  \\ 0 & \cos\theta & -\sin\theta & 0 \\ 0 & \sin\theta & \cos\theta & 0 \\
\frac{1}{2}(|z|-|z|^{-1}) & 0 & 0 & \frac{1}{2}(|z|+|z|^{-1}) \end{pmatrix}
\end{cases}
\]
We have chosen $f$ to be the unique map such that $f \circ d = p\circ \mathfrak{f}$ where $\mathfrak{f}$ is the natural embedding of $\mathcal{A}_0$ into $\mathfrak{S}$ given by $(o, \iota) = \left[\mathfrak{f}(o, \iota), \left(\begin{pmatrix}1 \\ 0 \end{pmatrix},\begin{pmatrix}0 \\ 1 \end{pmatrix}\right)\right]$. Note that $\mathfrak{f}( (o, \iota)\cdot z) = \mathfrak{f}(o, \iota) \cdot \begin{pmatrix} z & 0 \\ 0 & z^{-1}\end{pmatrix}$, therefore, $\mathfrak{f}$ is associated with the following embedding of Lie groups: \[\tilde{\mathfrak{f}}:\begin{cases} \C^*\rightarrow SL(2,\C)\\
z \rightarrow \begin{pmatrix}z & 0 \\ 0 & z^{-1} \end{pmatrix}\end{cases}.\] The fact that $f \circ d = p\circ \mathfrak{f}$ can be checked using a pair of compatible trivializations (and the associated trivializations for $\mathcal{S}$ and $T\mathcal{M}\otimes\C$) in which it follows from the equality (true for every $C_1,C_2\in \C^2$):
\begin{align*}
\tilde{p}\left(C_1, C_2\right)=&\left(i_0\left(\frac{C_1\otimes \overline{C_1} + C_2\otimes \overline{C_2}}{\sqrt{2}}\right),i_0\left(\frac{C_1\otimes \overline{C_2}+C_2\otimes\overline{C_1}}{\sqrt{2}}\right),\right.\\
& \left.-i_0\left(\frac{C_1\otimes\overline{C_2}-C_2\otimes\overline{C_1}}{\sqrt{2}i}\right), i_0\left(\frac{C_1\otimes\overline{C_1}-C_2\otimes\overline{C_2}}{\sqrt{2}}\right)\right).\end{align*} Note that $f$ is unique since $d$ is surjective.

We define $H:= \tilde{f}(\C^*)$ which is a commutative embedded Lie subgroup of $SO^+(1,3)$ and $\mathfrak{h}$ its Lie algebra (it is an abelian Lie subalgebra of $so(1,3)$). In order to apply remark \ref{remarkKillingForm} (and proposition \ref{pullBack2}), we want to prove that the Killing form of $so(1,3)$ is non degenerate and that its restriction to $\mathfrak{h}$ is also non degenerate.

It is easier to check this using the isomorphism of real Lie algebras $\dd_{Id} \tilde{p}$ since we have $(\dd_{Id} \tilde{p})^{-1}(\mathfrak{h}) = \left\{ \begin{pmatrix} z & 0\\ 0 & -z \end{pmatrix}, z\in \C\right\}$ (this follows from the equality $\tilde{f}(z^2) = \tilde{p}\begin{pmatrix}z & 0\\ 0 & z^{-1} \end{pmatrix}$). Note that we consider $sl(2,\C)$ as a real Lie algebra here, we will denote it by $sl(2,\C)_\R$ to emphasize this fact. The Killing form of $sl(2,\C)_\R$ is $B(M,N) = 8\Re(\tr(MN))$. Indeed, if $A_\C$ is the matrix of $ad_M\circ ad_N$ in the $\C$-basis $E_1 = \begin{pmatrix}1 & 0\\ 0 & -1\end{pmatrix}, E_2 = \begin{pmatrix}0 & 1 \\0 & 0\end{pmatrix}, E_3 = \begin{pmatrix}0 & 0 \\1 & 0\end{pmatrix}$, the matrix of $ad_M\circ ad_N$ in the $\R$ basis $(E_1, E_2, E_3, iE_1, iE_2, iE_3)$ is $A_\R:=\begin{pmatrix}\Re(A_{\C}) & -\Im(A_{\C})\\ \Im(A_{\C}) & \Re(A_{\C}) \end{pmatrix}$. Therefore $\tr(A_\R) = 2\Re(\tr(A_\C))$. We use this fact to deduce the expression of the Killing form of $sl(2,\C)_\R$ from the classical expression of the Killing form of $sl(2,\C)_{\C}$ (complex Lie algebra). Let $M\in sl(2,\C)_\R$, we assume that for all $N\in sl(2,\C)_R$, $B(M,N) = 0$. In particular, $B(M,\overline{M}^T) = 0$ and $B(M, -i\overline{M}^T)= 0$. Therefore $\tr(M\overline{M}^T) = 0$ and $M = 0$. We conclude that $B$ is non degenerate. Now let $M = \begin{pmatrix} z & 0\\ 0 & -z \end{pmatrix}$ (with $z\in \C$) be such that for all $N$ of the form $\begin{pmatrix} z' & 0 \\ 0 & -z'\end{pmatrix}$ (with $z'\in \C$), $B(M, N) = 0$. Again, we have $B\left(M,\overline{M}^T\right) = 0$ and $B\left(M, -i\overline{M}^T\right)= 0$ and therefore $M = 0$. We conclude that the restriction of $B$ to $(\dd_{Id} \tilde{p})^{-1}(\mathfrak{h})$ is non degenerate. A direct computation shows that the orthogonal of $(\dd_{Id} \tilde{p})^{-1}(\mathfrak{h})$ is $\left\{ \begin{pmatrix} 0 & a \\ b & 0\end{pmatrix}, a,b\in \C\right\}$. Then by remark \ref{remarkKillingForm} we can use proposition \ref{pullBack2} to define a principal connection $\omega$ on $\mathcal{N}_0$.
Finally we can use proposition \ref{pullBackCovering} to define a principal connection $\omega$ on $\mathcal{A}_0$, using the covering map $d$. Note that an equivalent definition is given by applying proposition \ref{pullBack2} for the natural embedding $\mathfrak{f}:\mathcal{A}_0 \rightarrow \mathfrak{S}$ and the orthogonal projection (with respect to the Killing form) on $\left\{ \begin{pmatrix} z & 0\\ 0 & -z \end{pmatrix}, z\in \C\right\}$ (the equivalence between the two definitions follows from the equality $f\circ d = p\circ \mathfrak{f}$).

\begin{remark}
\label{computationOmegaA}
Using the second definition of the connection on $\mathcal{A}_0$ and remark \ref{computationOmega}, we have that for $(o,\iota)$ a smooth section of $\mathcal{A}_0$ around $x_0\in \mathcal{M}$ and $X\in T_{x_0}\mathcal{M}$, 
\begin{align*}
\omega_{(o(x_0),\iota(x_0))}(\dd_{x_0}(o,\iota)(X)) &= \pi^{\perp}_{\mathfrak{h}}\begin{pmatrix}a(X) & c(X)\\b(X) & -a(X) \end{pmatrix} \\
&= \begin{pmatrix} a(X) & 0 \\ 0 & -a(X)\end{pmatrix}
\end{align*} where $a$,$b$ and $c$ are the same as in remark \ref{computationNabla} and can be computed using only $(l,n,m,\overline{m}):=d(o,\iota)$ and the Levi Civita connection. We have denoted the Killing-orthogonal projection on $(\dd_{Id}\tilde{p})^{-1}(\mathfrak{h})$ by $\pi^{\perp}_{\mathfrak{h}}$. Also note that in this equality, the Lie algebra of $\C^*$ has been identified with the Lie algebra of $\tilde{\mathfrak{f}}(C^*)$ (see proposition \ref{pullBack2}). Without the identification, we simply get $\omega_{(o(x_0),\iota(x_0))}(\dd_{x_0}(o,\iota)(X)) = a(X)$.
\end{remark}

We can use this principal connection to define a connection $\nabla$ on $\mathcal{B}(s,w)$ by proposition \ref{principalToLinear}. Now we give a way to compute concretely $\nabla u$ for $u$ a local smooth section of $\mathcal{B}(s,w)$.
\begin{prop}
\label{computationGHPConnection}
Let $(o, \iota)$ be a local smooth section of $\mathcal{A}_0$. Let $u$ be a local smooth section of $\mathcal{B}(s,w)$ on an open set $U$ of $\mathcal{M}$ and let $u_1: U\rightarrow \C$ be such that $u = [(o,\iota), u_1]$. Then for $x_0\in U$ and $X\in \C\otimes_\R T_{x_0}\mathcal{M}$
\[
\nabla_X u = [(o,\iota)(x_0), -\left((w+s)a(X) + (w-s)\overline{a(\overline{X})}\right)u_1(x_0) + X(u_1)]
\]
where $a(X) = \epsilon(\nabla_X o, \iota) = \epsilon(\nabla_X \iota, o)$ as previously.
\end{prop}
\begin{proof}
By $\C$ linearity of both sides, it is enough to prove the result for real vectors.
Let $x_0 \in U$ and $X\in T_{x_0}\mathcal{M}$ (real vector space). By corollary \ref{computationLinearGeneral} and remark \ref{computationOmegaA}, we have:
\begin{align*}
\nabla_X u &= [(o(x_0), \iota(x_0)), X(u_1)+\dd_{1}(\rho_{w,s})\omega(\dd_{x_0}(o,\iota)(X)) u_1(x_0)]\\
&= [(o(x_0), \iota(x_0)), X(u_1)+\dd_{1}(\rho_{w,s})(a(X))u_1(x_0)]\\
&= [(o(x_0), \iota(x_0)), X(u_1)-((s+w)a(X)+(w-s)\overline{a(\overline{X})})u_1(x_0)]
\end{align*}
\end{proof}

\begin{remark}
It is reassuring to check that for $u \in \mathcal{B}(s,w)$ and $v \in \mathcal{B}(s',w')$, if we fix $(o, \iota)$ a local smooth section of $\mathcal{A}_0$ near $x_0$ and write $u = [(o,\iota), u_1]$, $v = [(o,\iota), v_1]$, $u\otimes v = [(o,\iota), u_1v_1]$, and $\nabla_X u\otimes v = [(o, \iota), w_1]$ we have:
\begin{align*}
w_1=&\left(-(w+w'+s+s')a(X) - (w+w'-(s+s)')\overline{a(\overline{X})}\right)u_1(x_0)v_1(x_0) + X(u_1v_1) \\
=& \left(\left(-(w+s)a(X) - (w-s)\overline{a(\overline{X})}\right)u_1(x_0) + X(u_1)\right)v_1(x_0)\\
& + \left(\left(-(w'+s')a(X) - (w'-s')\overline{a(\overline{X})}\right)v_1(x_0) + X(v_1)\right) u_1(x_0)\\
\end{align*} 
and therefore $\nabla_X u\otimes v = (\nabla_X u) \otimes v + u\otimes (\nabla_X v)$. There is a compatibility between the connection acting on different spin weighted bundles and the isomorphism $\mathcal{B}(s,w)\otimes\mathcal{B}(s',w')=\mathcal{B}(s+s',w+w')$.
\end{remark}

\subsubsection{GHP operators}
If $u$ is a smooth section of $\mathcal{B}(s,w)$, then $\nabla u$ is a smooth section of $T^*_{\C}\mathcal{M}\otimes\mathcal{B}(s,w)$. We can define the operators which map $\nabla u$ to its spin weighted components (defined in remark \ref{swCompSwValuedCospinor}). Equivalently these operators can be seen as contraction of the spin weighted 2-cospinor $\nabla u$ with the spin weighted spinors $o$,$\iota$, $\overline{o}$ and $\overline{\iota}$. These operators are called GHP operators.
In this subsection, we denote by $\boldsymbol{o}$ the first projection of $\mathcal{A}_0$ and $\boldsymbol{\iota}$ the second projection. We also use the notation $\boldsymbol{l} = \boldsymbol{o}\otimes\overline{\boldsymbol{o}}$, $\boldsymbol{n}=\boldsymbol{\iota}\otimes\overline{\boldsymbol{\iota}}$ and $\boldsymbol{m}=\boldsymbol{o}\otimes\overline{\boldsymbol{\iota}}$.
We define the operators
\[
\begin{array}{cc}
\text{\th}:\begin{cases}
\Gamma(\mathcal{B}(s,w))\rightarrow \Gamma(\mathcal{B}(s,w+1))\\
u \mapsto \nabla_{\boldsymbol{l}}u
\end{cases}
&
\text{\th}':\begin{cases}
\Gamma(\mathcal{B}(s,w))\rightarrow \Gamma(\mathcal{B}(s,w-1))\\
u \mapsto \nabla_{\boldsymbol{n}}u
\end{cases}\\

\text{\dh}:\begin{cases}
\Gamma(\mathcal{B}(s,w))\rightarrow \Gamma(\mathcal{B}(s+1,w))\\
u \mapsto \nabla_{\boldsymbol{m}}u
\end{cases}
&
\text{\dh}':\begin{cases}
\Gamma(\mathcal{B}(s,w))\rightarrow \Gamma(\mathcal{B}(s-1,w))\\
u \mapsto \nabla_{\overline{\boldsymbol{m}}}u
\end{cases}
\end{array}
\]
\begin{remark}
\label{expressionGHP}
We can use proposition \ref{computationGHPConnection} to compute \text{\th}, \text{\th$'$}, \text{\dh} and \text{\dh$'$} in a local trivialization given by a smooth local section $(o, \iota)$ of $\mathcal{A}_0$. Indeed, if $u\in \mathcal{B}(s,w)$ writes $u = [(o,\iota), u_1]$ we have:
\begin{align*}
(\text{\th} u)_1 &= \left(-(w+s)a(l) - (w-s)\overline{a(l)}\right)u_1 + l(u_1)\\
(\text{\th'}u)_1 &= \left(-(w+s)a(n) - (w-s)\overline{a(n)}\right)u_1 + n(u_1)\\
(\text{\dh}u)_1 &= \left(-(w+s)a(m) - (w-s)\overline{a(\overline{m})}\right)u_1 + m(u_1)\\
(\text{\dh'}u)_1 &= \left(-(w+s)a(\overline{m}) - (w-s)\overline{a(m)}\right)u_1 + \overline{m}u_1 \quad.
\end{align*}
\end{remark}

\begin{prop}
There exists a unique element $\boldsymbol{b(\overline{m})}$ of $\mathcal{B}(0,1)$ such that for every local smooth section $o,\iota$ of $\mathcal{A}_0$ defined on an open set $U$, $\boldsymbol{b(\overline{m})}_{|_U} = [(o,\iota), b(\overline{m})]$ (where $b(\overline{m}) := -\epsilon(\nabla_{\iota\otimes \overline{o}}o, o)$).
\end{prop}
\begin{proof}
The only thing to prove is that the definition does not depend of the choice of the local section $(o,\iota)$. Let $(o,\iota)$ be a local smooth section of $\mathcal{A}_0$ on an open set $U$. Let $z:U\rightarrow \C^*$ be a smooth map and $(o', \iota') = (o, \iota)\cdot z = (zo, z^{-1}\iota)$ (every local smooth section on $U$ can be written in this form). We then have:
\begin{align*}
\epsilon(\nabla_{\iota'\otimes \overline{o}'}o', o') =& \epsilon(\nabla_{\overline{z}z^{-1}\iota\otimes \overline{o}}(zo), zo) \\
=& \epsilon(\overline{z}z^{-1}\overline{m}(z)o, zo) + \epsilon(\overline{z}\nabla_{\iota\otimes \overline{o}}o, zo) \\
=& z\overline{z}\epsilon(\nabla_{\iota\otimes \overline{o}}o, o)\\
=& \rho_{(0,1)}(z^{-1})\epsilon(\nabla_{\iota\otimes \overline{o}}o, o)
\end{align*}
Since $[(o,\iota), -\rho_{0,1}(z^{-1})\epsilon(\nabla_{\iota\otimes \overline{o}}o, o)] = [(o',\iota')\cdot z^{-1}, -\epsilon(\nabla_{\iota\otimes \overline{o}}o, o)]$ by definition of $\mathcal{B}(0,1)$, we finally have: 
\[ [(o',\iota'), -\epsilon(\nabla_{\iota'\otimes \overline{o}'}o', o')] = [(o,\iota), -\epsilon(\nabla_{\iota\otimes \overline{o}}o, o)]\]
\end{proof}
\begin{remark}
\label{multiplicationOperator}
Thanks to remark \ref{productSpinWeighted}, we can see $\boldsymbol{b(\overline{m})}$ as a multiplication operator from $\mathcal{B}(s,w)$ to $\mathcal{B}(s,w+1)$.
\end{remark}
\begin{remark}
\label{otherSWFunctions}
We similarly define the spin weighted functions $\boldsymbol{b(n)} \in \mathcal{B}(1,0)$, $\boldsymbol{c(m)}\in \mathcal{B}(0,-1)$ and $\boldsymbol{c(l)} \in \mathcal{B}(-1,0)$.
\end{remark}

\section{Definition of the Teukolsky operator}
\begin{definition}[Contraction operator]
Let $\phi \in (\mathcal{S}')^{\otimes n_0}$ ($n_0\in \N$). We define the operator $C_{i,j}$ (with $i<j$) such that in any $(s_0,s_1)$ basis of $\mathcal{S}$ with $\epsilon(s_0,s_1) = 1$:
\begin{align*}
(C_{i,j}\phi)(s_{k_1},..., \hat{s_{k_{i}}},..., \hat{s_{k_j}},...,s_{i_{n_0}}) =& \phi(s_{k_1},...,s_{k_{i-1}}, s_1, s_{k_{i+1}},..., s_{k_{j-1}}, s_0, s_{k_{j+1}}, ..., s_{k_{n_0}})\\
&-\phi(s_{k_1},...,s_{k_{i-1}}, s_0, s_{k_{i+1}},..., s_{k_{j-1}}, s_1, s_{k_{j+1}}, ..., s_{k_{n_0}})
\end{align*}
where $k_1, ... k_n \in \left\{0,1\right\}$ and $\hat{s_{k_{l}}}$ means that $s_{k_{l}}$ is skipped in the enumeration.
This definition does not depend on the chosen basis as long as it is normalized.
\end{definition}

\begin{definition}
Let $\phi \in \Gamma((\mathcal{S}')^{\odot n_0})$. We define the operator $D$ (Dirac operator) as $D\phi = C_{1,3}\nabla \phi \in \Gamma(\overline{\mathcal{S}}'\otimes(\mathcal{S}')^{\otimes n_0-1})$. Because the spinor is symmetric, we can replace $3$ in the definition by any index in $\left\{3,...,n_0+2\right\}$.
\end{definition}
\begin{remark}
Note that $D\phi$ is symmetric with respect to the last $n_0-1$ variables but the first one has a particular status.
\end{remark}

\begin{prop}
\label{SWDirac}
We have the following relations at the level of spin weighted components:
\begin{align}
(D\phi)(\overline{\boldsymbol{o}},\boldsymbol{o},...,\boldsymbol{o}) =& \left(\text{\dh'} + \boldsymbol{c(l)}\right)\phi(\boldsymbol{o},...,\boldsymbol{o})-\left(\text{\th}+n_0\boldsymbol{b(\overline{m})}\right) \phi(\boldsymbol{\iota}, \boldsymbol{o},...,\boldsymbol{o}) \label{SWDirac1}\\
(D\phi)(\overline{\boldsymbol{\iota}},\boldsymbol{o},...,\boldsymbol{o}) =& \left(\text{\th'}+\boldsymbol{c(m)}\right)\phi(\boldsymbol{o},...,\boldsymbol{o}) -\left(\text{\dh} + n_0\boldsymbol{b(n)}\right)\phi(\boldsymbol{\iota}, \boldsymbol{o},...,\boldsymbol{o}) \label{SWDirac2}
\end{align}
where $\boldsymbol{c(l)}$, $\boldsymbol{c(m)}$, $\boldsymbol{b(\overline{m})}$ and $\boldsymbol{b(n)}$ are seen as multiplications operators on spin-weighted functions (see remark \ref{multiplicationOperator} and \ref{otherSWFunctions})
\end{prop}
\begin{proof}
Since both sides of the equalities are spin weighted functions of the same weight ($(\frac{n_0}{2}-1, \frac{n_0}{2})$ for the first and $(\frac{n_0}{2},\frac{n_0}{2}-1)$ for the second), it is enough to check the equality in a local trivialization near each point. We do it for the first equality (the second is similar). Let $x_0 \in \mathcal{M}$, let $(o,\iota)$ be a local section of $\mathcal{A}_0$. Thanks to the bold notation, it is easy to compute the components in local trivializations associated to this local section:
\begin{align*}
(D\phi)(\overline{o},o,...,o) =& (\nabla_{\iota\otimes\overline{o}}\phi)(o,...,o) - (\nabla_{o\otimes\overline{o}}\phi)(\iota,o,...,o)\\
=& \overline{m}(\phi(o,...,o))-n_0\phi(\nabla_{\overline{m}} o, o,...,o) - l(\phi(\iota, o,...,o)) + \phi(\nabla_{l}\iota, o,...,o) + (n_0-1) \phi(\iota, \nabla_l o,..., o)\\
=& (\overline{m} - n_0 a(\overline{m}) + c(l))(\phi(o,...,o)) - (l+n_0 b(\overline{m})-(n_0-2)a(l))(\phi(\iota, o, ...,o)) \\
&+ (n_0-1)b(l)\phi(\iota, \iota, o, ..., o)
\end{align*}
By remark \ref{vanishingSpinCoeff}, we have $b(l) = 0$. Moreover, since $\phi(\boldsymbol{o},...,\boldsymbol{o})\in \Gamma\left(\mathcal{B}\left(\frac{n_0}{2}, \frac{n_0}{2}\right)\right)$, the expression of \text{\dh'} on this bundle expressed in the local trivialization induced by $(o,\iota)$ is $\overline{m} - n_0 a(\overline{m})$ by remark \ref{expressionGHP}. Similarly, the fact that $\phi(\iota, o, ...,o)\in \Gamma\left(\mathcal{B}\left(\frac{n_0}{2}-1, \frac{n_0}{2}-1\right)\right)$ gives the local expression $l-(n_0-2)a(l)$ for \th. By definition, $c(l)$ and $b(\overline{m})$ are the local expressions for $\boldsymbol{c(l)}$ and $\boldsymbol{b(\overline{m})}$ in the local trivialization induced by $(o,\iota)$. Therefore, we have the desired equality.
\end{proof}

\begin{prop}
We have the following relations at the level of spin weighted components:
\begin{align*}
(\text{\th}+n_0\boldsymbol{b(\overline{m})}+\overline{\boldsymbol{b(\overline{m})}})(D\phi)(\overline{\boldsymbol{\iota}},\boldsymbol{o}, ..., \boldsymbol{o})-(\text{\dh}+\overline{\boldsymbol{c(l)}}+n_0\boldsymbol{b(n)})(D\phi)(\overline{\boldsymbol{o}},...,\boldsymbol{o}) &= G_{\frac{n_0}{2}}\phi(\boldsymbol{o},...,\boldsymbol{o})
\end{align*}
where
\[
G_s = (\text{\th}+2s\boldsymbol{b(\overline{m})}+\overline{\boldsymbol{b(\overline{m})}})(\text{\th'}+\boldsymbol{c(m)}) - (\text{\dh} + \overline{\boldsymbol{c(l)}} + 2s \boldsymbol{b(n)})(\text{\dh'}+\boldsymbol{c(l)})
\]
is a smooth differential operator acting on the bundle $\mathcal{B}(s,s)$
\end{prop}
\begin{proof}
We apply $(\text{\th}+n_0\boldsymbol{b(\overline{m})}+\overline{\boldsymbol{b(\overline{m})}})$ on the left of \eqref{SWDirac2} (in proposition \ref{SWDirac}) and $(\text{\dh} + \overline{\boldsymbol{c(l)}} + n_0 \boldsymbol{b(n)})$ on the left of \eqref{SWDirac1} and we take the difference. Then, we use the remarkable relation
\begin{align}
\label{magicalRelation}
(\text{\th}+n_0\boldsymbol{b(\overline{m})}+\overline{\boldsymbol{b(\overline{m})}})\left(\text{\dh} + n_0\boldsymbol{b(n)}\right) - (\text{\dh} + \overline{\boldsymbol{c(l)}} + n_0 \boldsymbol{b(n)})\left(\text{\th}+n\boldsymbol{b(\overline{m})}\right) = 0
\end{align}
as an operator acting on the bundle $\mathcal{B}\left(\frac{n_0}{2}, \frac{n_0}{2}\right)$. Therefore, the component $\phi(\boldsymbol{\iota}, \boldsymbol{o}, ..., \boldsymbol{o})$ disappears and we have the result. The last thing to prove is the relation \eqref{magicalRelation}.
It is enough to prove it in a local trivialization provided by a smooth local section $(o,\iota)$ of $\mathcal{A}_0$ near each point $x_0 \in \mathcal{M}$. Therefore, we want to prove that:
\begin{align*}
\left(l-(n_0-1)a(l)+\overline{a(l)}+n_0b(\overline{m})+\overline{b(\overline{m})}\right)\left(m-(n_0-2)a(m) + n_0b(n)\right)&\\ - (m-(n_0-1)a(m)-\overline{a(\overline{m})} + \overline{c(l)} + n_0 b(n))\left(l-(n_0-2)a(l)+n_0b(\overline{m})\right) &= 0
\end{align*}
If we use remark \ref{computationNabla} to replace the occurences of $a$,$b$ and $c$ by spin coefficients, we rewrite this as:
\begin{align*}
(l-(n_0-1)\epsilon+\overline{\epsilon}-n_0\rho-\overline{\rho})\left(m-(n_0-2)\beta - n_0\tau\right)&\\
 - (m-(n_0-1)\beta-\overline{\alpha} + \overline{\pi} - n_0 \tau)\left(l-(n_0-2)\epsilon-n_0\rho\right)&=0
\end{align*}
 We see that this relation is a particular case of a relation between spin coefficients of a null tetrad with $l$ and $n$ in principal directions introduced by Teukolsky in \cite{teukolsky1973perturbations} (equation (2.11)). It is the case $p = n_0-2$ and $q = -n_0$.
\end{proof}

\begin{definition}
We define the Teukolsky operator $T_{s} := 2G_s+4(s-1)\left(s-\frac{1}{2}\right)\boldsymbol{\Psi_2}$. It is naturally a differential operator on $\mathcal{B}(s,s)$.
\end{definition}

\subsection{Formula for the Teukolsky operator on Kerr in a trivialization provided by the Kinnersley tetrad}
We consider the local section $(o_m, \iota_m)$ of $\mathcal{A}_0$ defined on $U:=\R_t\times (r_0, +\infty)\times \mathbb{S}^2\setminus\left\{\phi = 0\right\}$. We recall that by definition $d((o_m,\iota_m))= (l,n,m)$ where $l,n,m,\overline{m}$ is the Kinnersley tetrad defined by \eqref{principalVectors}, \eqref{principalVectors2} and \eqref{vectorM}. Note that $(-o_m, -\iota_m)$ has also the property $d((-o_m, -\iota_m)) = (l,n,m)$ (it is the only other section of $\mathcal{A}_0$ defined on $U$ with this property). However, the expression of $T_s$ in a local trivialization provided by a section $(o,\iota)$ only depend on $d(o,\iota)$ (this follows from the fact that it is the case for operators \th, \th', {\dh} and {\dh'} and for the spin weighted functions $\boldsymbol{b(\overline{m})}$, $\boldsymbol{b(n)}$ and $\boldsymbol{c(l)}$). Therefore, it is correct to speak about the formula for the Teukolsky operator in a trivialization provided by the Kinnersley tetrad without further precision. But to fix the ideas, we consider here the differential operator $(T_s)_m$ ($T_s$ written in the local trivialization provided by $A_m$, it is therefore a differential operator on $U$). We use the expression of \th, \th', {\dh} and {\dh} computed in remark \ref{expressionGHP} and the definition of $\boldsymbol{b(\overline{m})}$, $\boldsymbol{b(n)}$ and $\boldsymbol{c(l)}$ to find for $s\in \frac{1}{2}\Z$:
\begin{align*}
(T_s)_m =& 2\left(l-(2s-1)a(l)+\overline{a(l)}+2sb(\overline{m})+\overline{b(\overline{m})}\right)(n - 2s a(n) + c(m))\\
&-2\left(m-(2s-1)a(m)-\overline{a(\overline{m})}+\overline{c(l)}+2sb(n)\right)(\overline{m}-2sa(\overline{m})+c(l))\\
& + 4(s-1)\left(s-\frac{1}{2}\right)\Psi_2
\end{align*}
where $a$, $b$, $c$ and $\Psi_2$ have to be computed with respect to the tetrad $(l,n,m,\overline{m})$ (see remark \ref{computationNabla}). The computation of these coefficients is done in \cite{chandrasekhar1998mathematical} (chapter 6, section 56, equation (175) and (180) but with the opposite sign convention for $\Psi_2$) and we finally find:
\begin{align*}
(r^2+a^2\cos^2(\theta))(T_s)_m =& \left(\frac{(r^2+a^2)^2}{\Delta_r} - a^2 \sin^2\theta\right)\partial_t^2 + \frac{4Mar}{\Delta_r}\partial_t\partial_\phi + \left(\frac{a^2}{\Delta_r} - \frac{1}{\sin^2\theta}\right)\partial_\phi^2 - \Delta_r^{-s}\partial_r\left(\Delta_r^{s+1}\partial_r\right)\\
& - \frac{1}{\sin\theta}\partial_\theta \left(\sin\theta \partial_\theta\right) - 2s\left(\frac{a(r-M)}{\Delta_r}+\frac{i\cos\theta}{\sin^2\theta}\right)\partial_\phi - 2s\left(\frac{M(r^2-a^2)}{\Delta_r} - r - i a \cos \theta\right)\partial_t\\
& + \left(s^2 \cot^2\theta - s\right)
\end{align*}

%*** Préciser qu'on utilise la convention (1,-1,-1,-1)
%*** Peut-être préciser le sens de $\mathbb{S}^2\setminus\left\{\phi = 0\right\}$
\bibliography{biblio}
\end{document}